\documentclass[twoside,12pt,a4paper,leqno,english]{amsart}

\usepackage{amsmath,amsfonts,amsthm,enumitem}
\usepackage[colorlinks=true,linkcolor=blue,citecolor=blue,%
filecolor=blue,urlcolor=blue,pageanchor=true,plainpages=false,%
linktocpage]{hyperref}
\usepackage[english]{babel}

\numberwithin{equation}{section}
\newtheorem{thm}{Theorem}[section]
\newtheorem{prop}[thm]{Proposition}

\newtheorem{lem}[thm]{Lemma}
\theoremstyle{definition}
\newtheorem{defn}[thm]{Definition}
\newtheorem{rem}[thm]{Remark}

\newcommand{\N}{\mathbb{N}}
\newcommand{\Z}{\mathbb{Z}}
\newcommand{\R}{\mathbb{R}}
\newcommand{\Rb}{\overline{\R}}

\newcommand{\cl}[1]{\overline{#1}}

\newcommand{\esssup}{\operatornamewithlimits{\mathrm{ess\,sup}}}
\newcommand{\essinf}{\operatornamewithlimits{\mathrm{ess\,inf}}}
\newcommand{\degc}{\mathrm{ind}}
\makeatletter
\newcommand{\mylabel}[2]{#2\def\@currentlabel{#2}\label{#1}}
\makeatother
\newcommand{\ABi}{$(h_1)$}
\newcommand{\ABii}{$(h_2)$}
\newcommand{\ABiii}{$(h_3)$}

\newcommand{\uc}{$(UC)$}
\newcommand{\um}{$(UM)$}
\newcommand{\hci}{$(C)$}
\newcommand{\uni}{$(UN)$}
\newcommand{\hni}{$(N)$}
\newcommand{\mc}{$(MC)$}
\newcommand{\hnii}{$(M)$}
\newcommand{\hniii}{$(B)$}

\renewcommand{\rho}{\varrho}
\renewcommand{\theta}{\vartheta}
\newcommand{\au}[1]{\textsc{#1}}
\newcommand{\titleart}[1]{\textrm{#1}}
\newcommand{\jour}[1]{\textit{#1}}
\newcommand{\volart}[1]{\textbf{#1}}
\newcommand{\no}[1]{\textit{no.}~{#1}}
\begin{document}

%TOPMATTER

\title[Nontrivial solutions of quasilinear elliptic 
equations]{Nontrivial
solutions of quasilinear elliptic \\
equations with natural growth term}

\author{Marco Degiovanni}
\address{Dipartimento di Matematica e Fisica\\
         Universit\`a Cattolica del Sacro Cuore\\
         Via dei Musei 41\\
         25121 Bre\-scia, Italy}
\email{marco.degiovanni@unicatt.it}
\thanks{The authors are member of the 
        Gruppo Nazionale per l'Analisi Matematica, la Probabilit\`a 
				e le loro Applicazioni (GNAMPA) of the 
				Istituto Nazionale di Alta Matematica (INdAM)}
\author{Alessandra Pluda}
\address{Fakult\"{a}t f\"{u}r Mathematik \\
          Universit\"{a}t Regensburg\\
          Universit\"{a}tstrasse 31\\
          93053 Regensburg, Germany}
\email{Alessandra.Pluda@mathematik.uni-regensburg.de}

\dedicatory{Dedicated to Gianni Gilardi}



%END TOPMATTER

%--------------------------------------------------------------------

%
\begin{abstract}
We prove the existence of multiple solutions for a quasilinear
elliptic equation containing a term with natural growth, under
assumptions that are invariant by diffeomorphism.
To this purpose we develop an adaptation of degree theory.
\end{abstract}
\maketitle %AMSART

%--------------------------------------------------------------------

\section{Introduction}
Consider the quasilinear elliptic problem
\begin{equation}
\label{eq:qe}
\left\{
\begin{array}{ll}
- \mathrm{div}\left[A(x,u)\nabla u\right]
+ B(x,u)|\nabla u|^2 = g(x,u)
&\qquad\text{in $\Omega$}\,,\\
\noalign{\medskip}
u=0
&\qquad\text{on $\partial\Omega$}\,,
\end{array}
\right.
\end{equation}
where $\Omega$ is a bounded, connected and open subset of $\R^n$ and
\[
A, B, g:\Omega\times\R\rightarrow\R
\]
are Carath\'eodory functions such that:
\emph{
\begin{enumerate}
\item[\mylabel{ABi}{\ABi}]
for every $R>0$ there exist $\beta_R> 0$ and $\nu_R>0$ satisfying
\[
\nu_R \leq A(x,s) \leq \beta_R\,,\qquad
|B(x,s)| \leq \beta_R\,,\qquad
|g(x,s)| \leq \beta_R\,|s| \,,
\]
for a.e. $x\in\Omega$ and every $s\in\R$ with $|s|\leq R$.
\end{enumerate}
}
The existence of a weak solution 
$u\in W^{1,2}_0(\Omega)\cap L^\infty(\Omega)$
of~\eqref{eq:qe}, namely a solution of
\begin{multline}
\label{eq:AB}
\int_\Omega \left[A(x,u)\nabla u\cdot\nabla v
+ B(x,u)|\nabla u|^2\,v\right]\,dx = 
\int_\Omega g(x,u)v\,dx \\
\text{for any 
$v\in W^{1,2}_0(\Omega)\cap L^\infty(\Omega)$}\,,
\end{multline}
follows from the results 
of~\cite{boccardo_murat_puel1984, boccardo_murat_puel1988},
provided that a suitable \emph{a priori} $L^\infty$-estimate holds,
possibly related to the existence of a pair of sub-/super-solutions
(see 
e.g.~\cite[Th\'eor\`eme~2.1]{boccardo_murat_puel1984}
and~\cite[Theorems~1 and~2]{boccardo_murat_puel1988}).
Moreover, each weak solution 
$u\in W^{1,2}_0(\Omega)\cap L^\infty(\Omega)$ is locally 
H\"{o}lder continuous in $\Omega$ (see 
e.g.~\cite[Theorem~VII.1.1]{giaquinta1983}).
\par
The existence of multiple solutions, in the semilinear case
and under suitable regularity assumptions,
has also been proved for instance in~\cite{amann_crandall1978}.
Here we are interested in the existence of multiple 
nontrivial solutions, as~\ref{ABi} implies that $g(x,0)=0$,
under assumptions that do not imply an \emph{a priori}
$W^{1,\infty}$-estimate.
Let us state our main result.
\begin{thm}
\label{thm:main}
Assume~\ref{ABi} and also that:
\begin{enumerate}
\item[\mylabel{ABii}{\ABii}]
there exist $\underline{M}<0<\overline{M}$ such that
\[
g(x,\underline{M}) \geq 0 \geq g(x,\overline{M})
\qquad\text{for a.e. $x\in\Omega$}\,;
\]
\item[\mylabel{ABiii}{\ABiii}]
the function $g(x,\cdot)$ is differentiable at $s=0$ for 
a.e. $x\in\Omega$.
\end{enumerate}
\par
Consider the eigenvalue problem
\begin{equation}
\label{eq:ABlin}
\left\{
\begin{array}{ll}
- \mathrm{div}\left[A(x,0)\nabla u\right] - D_sg(x,0)u
= \lambda u
&\qquad\text{in $\Omega$}\,,\\
\noalign{\medskip}
u=0
&\qquad\text{on $\partial\Omega$}\,,
\end{array}
\right.
\end{equation}
and denote by $(\lambda_k)$, $k\geq 1$, the sequence of the 
eigenvalues repeated according to multiplicity.
\par
If there exists $k\geq 2$ with $\lambda_k<0<\lambda_{k+1}$
and $k$ even, then problem~\eqref{eq:AB} admits at least
three nontrivial solutions $u_1, u_2, u_3$ in 
$W^{1,2}_0(\Omega)\cap L^\infty(\Omega)\cap C(\Omega)$ with
\[
\text{$u_1<0$ in $\Omega$, \quad $u_2>0$ in $\Omega$,\quad
$u_3$ sign-changing}\,.
\]
\end{thm}
If $A$ is constant and $B=0$, the result is essentially
contained in~\cite{ambrosetti_lupo1984, struwe1982},
which in turn developed previous results
of~\cite{ambrosetti_mancini1979}.
Actually, in those papers it is enough to assume that 
$\lambda_2<0$,
because in that case~\eqref{eq:AB} is the Euler-Lagrange
equation of a suitable functional and variational methods,
e.g. Morse theory, can be applied.
\par
In our case there is no functional and also degree theory 
arguments cannot be applied in a standard way, because
of the presence of the term $B(x,u)|\nabla u|^2$.
Let us point out that our assumptions do not imply that
the solutions $u$ of~\eqref{eq:AB} belong to 
$W^{1,\infty}(\Omega)$, so that the natural growth term
$B(x,u)|\nabla u|^2$ plays a true role.
\par
Let us also mention that our statement has 
an invariance property.
\begin{rem}
Let $\varphi:\R\rightarrow\R$ be an increasing
$C^2$-diffeomorphism with $\varphi(0)=0$.
Then the following facts hold:
\begin{enumerate}
\item[$(a)$]
the functions $A$, $B$ and $g$ satisfy the
assumptions of Theorem~\ref{thm:main} if and only
$A^\varphi$, $B^\varphi$ and $g^\varphi$, defined as
\begin{alignat*}{3}
&A^\varphi(x,s)
&&=
(\varphi'(s))^2\,A(x,\varphi(s))\,,\\
&B^\varphi(x,s)
&&=
\varphi'(s) \varphi''(s) A(x,\varphi(s)) 
+ (\varphi'(s))^3 B(x,\varphi(s))\,,\\
&g^\varphi(x,s) 
&&= 
\varphi'(s) g(x,\varphi(s))\,,
\end{alignat*}
do the same;
in particular, we have
\[
A^\varphi(x,0) = (\varphi'(0))^2\,A(x,0)\,,\qquad
D_sg^\varphi(x,0) = (\varphi'(0))^2\,D_sg(x,0)\,;
\]
\item[$(b)$]
a function
$u\in W^{1,2}_0(\Omega)\cap L^\infty(\Omega)$
is a solution of~\eqref{eq:AB} if and only if
$\varphi^{-1}(u)$ is a solution of~\eqref{eq:AB}
with $A$, $B$ and $g$ replaced by
$A^\varphi$, $B^\varphi$ and $g^\varphi$, respectively.
\end{enumerate}
\end{rem}
From the definition of $B^\varphi$ we see that the term
$B(x,u)|\nabla u|^2$ cannot be omitted in the equation,
if we want to ensure this kind of invariance.
\par
When~\eqref{eq:AB} is the Euler-Lagrange equation
associated to a functional, the question of invariance
under suitable classes of diffeomorphisms has been already
treated in~\cite{solferino_squassina2012}, where it is shown
that problems with degenerate coercivity can be reduced, 
in some cases, to coercive problems.
\par
In the next sections we develop an adaptation of degree
theory suited for our setting and then we prove 
Theorem~\ref{thm:main} by a degree argument.
Under assumptions that are not diffeomorphism-invariant,
a degree theory for quasilinear elliptic equations with
natural growth conditions has been already developed
in~\cite{almi_degiovanni2013}.
Here we find it more convenient to reduce the 
equation~\eqref{eq:AB} to a variational inequality
possessing as obstacles a pair of sub-/super-solutions,
according to an approach already considered for
instance in~\cite{saccon2014}.

%--------------------------------------------------------------------

\section{Topological degree in reflexive Banach spaces}
\label{sect:banach}
Let $X$ be a reflexive real Banach space.
\begin{defn}
\label{defn:S+}
A map $F:D\rightarrow X'$, with $D\subseteq X$, is said
to be \emph{of class~$(S)_+$} if, for every sequence $(u_k)$
in $D$ weakly convergent to some $u$ in $X$ with
\[
\limsup_k \, \langle F(u_k),u_k-u\rangle \leq 0\,,
\]
it holds $\|u_k-u\|\to 0$.
\par
More generally, if $T$ is a metrizable topological space,
a map $H:D\rightarrow X'$, with $D\subseteq X\times T$, is said
to be \emph{of class~$(S)_+$} if, for every sequence
$(u_k,t_k)$ in $D$ with
$(u_k)$ weakly convergent to $u$ in $X$,
$(t_k)$ convergent to $t$ in $T$ and
\[
\limsup_k \, \langle H_{t_k}(u_k),u_k-u\rangle \leq 0\,,
\]
it holds $\|u_k-u\|\to 0$ (we write $H_t(u)$
instead of $H(u,t)$).
\end{defn}
Assume now that $U$ is a bounded and open subset of $X$, 
$F:\overline{U}\rightarrow X'$ a bounded and
continuous map of class~$(S)_+$, $K$ a closed and convex
subset of $X$ and $\varphi\in X'$.
We aim to consider the variational inequality
\begin{equation}
\label{eq:vi}
\begin{cases}
u\in K \,,\\
\noalign{\medskip}
\langle F(u),v-u\rangle\geq \langle\varphi,v-u\rangle
&\qquad\forall v\in K\,.
\end{cases}
\end{equation}
\begin{rem}
It is easily seen that the set
\[
\left\{u\in \cl{U}:\,\,
\text{$u$ is a solution of~\eqref{eq:vi}}\right\}
\]
is compact (possibly empty).
\end{rem}
According 
to~\cite{browder1983, motreanu_motreanu_papageorgiou2014, 
skrypnik1994},
if the variational inequality~\eqref{eq:vi}
has no solution $u\in \partial{U}$, one can define
the topological degree
\[
\mathrm{deg}((F,K),U,\varphi)\in\Z\,.
\]
Let us recall some basic properties.
\begin{prop}
If~\eqref{eq:vi} has no solution $u\in \partial{U}$, then
\[
\mathrm{deg}((F,K),U,\varphi) = \mathrm{deg}((F-\varphi,K),U,0)\,.
\]
\end{prop}
\begin{prop}
\label{prop:traslX}
If~\eqref{eq:vi} has no solution $u\in \partial{U}$,
$u_0\in X$ and we set 
\begin{alignat*}{3}
&\widehat{U} &&=
\left\{u-u_0:\,\,u\in U\right\}\,,\\
&\widehat{K} &&=
\left\{u-u_0:\,\,u\in K\right\}\,,\\
&\widehat{F}(u) &&= F(u_0+u)\,,
\end{alignat*}
then
\[
\mathrm{deg}((\widehat{F},\widehat{K}),\widehat{U},\varphi) 
= \mathrm{deg}((F,K),U,\varphi)\,.
\]
\end{prop}
\begin{thm}
\label{thm:existencedegree}
If~\eqref{eq:vi} has no solution $u\in \cl{U}$, then
$\mathrm{deg}((F,K),U,\varphi) = 0$.
\end{thm}
\begin{thm}
\label{thm:deg1}
If~\eqref{eq:vi} has no solution $u\in \partial{U}$
and there exists $u_0\in K\cap U$ such that
\[
\langle F(v),v-u_0\rangle \geq \langle\varphi,v-u_0\rangle
\qquad\text{for any $v\in K\cap\partial{U}$}\,,
\]
then $\mathrm{deg}((F,K),U,\varphi) = 1$.
\end{thm}
\begin{thm}
\label{thm:add-exc}
If $U_0$ and $U_1$ are two disjoint open subsets of $U$
and~\eqref{eq:vi} has no solution 
$u\in \cl{U}\setminus(U_0\cup U_1)$, then
\[
\mathrm{deg}((F,K),U,\varphi) =
\mathrm{deg}((F,K),U_0,\varphi) + 
\mathrm{deg}((F,K),U_1,\varphi) \,.
\]
\end{thm}
\begin{defn}
Let $K_k$, $K$ be closed and convex subsets of $X$.
The sequence $(K_k)$ is said to be \emph{Mosco-convergent}
to $K$ if the following facts hold:
\begin{enumerate}
\item[$(a)$]
if $k_j\to\infty$, $u_{k_j}\in K_{k_j}$ for any $j\in\N$
and $(u_{k_j})$ is weakly convergent to $u$ in $X$,
then $u\in K$;
\item[$(b)$]
for every $u\in K$ there exist $\overline{k}\in\N$ and a sequence 
$(u_k)$ strongly convergent to $u$ in~$X$ with $u_k\in K_k$ 
for any $k\geq\overline{k}$.
\end{enumerate}
\end{defn}
\begin{thm}
\label{thm:homotopydegree}
Let $W$ be a bounded and open subset of $X\times[0,1]$,
$H:\overline{W}\rightarrow X'$ be
a bounded and continuous map of class~$(S)_+$
and $(K_t)$, $0\leq t\leq1$, be a family of closed
and convex subsets of $X$ such that, for every sequence $(t_k)$
convergent to $t$ in $[0,1]$, the sequence $(K_{t_k})$ is
Mosco-convergent to $K_t$.
\par 
Then the following facts hold:
\begin{enumerate}
\item[$(a)$]
the set of pairs $(u,t)\in\cl{W}$, satisfying
\begin{equation}
\label{eq:hvi}
\begin{cases}
u\in K_t\,,\\
\noalign{\medskip}
\langle H_t(u),v-u\rangle\geq \langle\varphi,v-u\rangle
\qquad\forall v\in K_t\,,
\end{cases}
\end{equation}
is compact (possibly empty);
\item[$(b)$]
if the problem~\eqref{eq:hvi} has no solution 
$(u,t)\in \partial_{X\times[0,1]}\,W$ and we set
\[
W_t = \left\{u\in X:\,\,(u,t)\in W\right\}\,,
\]
then $\mathrm{deg}((H_{t},K_t),W_t,\varphi)$ is independent of
$t\in[0,1]$.
\end{enumerate}
\end{thm}
\begin{proof}
If $(u_k,t_k)$ is a sequence in $\cl{W}$ constituted
by solutions of~\eqref{eq:hvi}, then up to a subsequence
$(u_k)$ is weakly convergent to some $u$ in $X$ 
and $(t_k)$ is convergent to some $t$ in~$[0,1]$.
Then $u\in K_t$ and there exists a sequence $(\hat{u}_k)$ 
strongly convergent to $u$ in $X$ with $\hat{u}_k\in K_{t_k}$.
It follows
\[
\begin{split}
\langle H_{t_k}(u_k),u_k-u\rangle 
&=
\langle H_{t_k}(u_k),\hat{u}_k-u\rangle 
+ \langle H_{t_k}(u_k),u_k-\hat{u}_k\rangle \\
&\leq
\langle H_{t_k}(u_k),\hat{u}_k-u\rangle 
+ \langle \varphi,u_k-\hat{u}_k\rangle \,,
\end{split}
\]
whence
\[
\limsup_k\,\langle H_{t_k}(u_k),u_k-u\rangle \leq 0\,.
\]
Then $\|u_k-u\|\to 0$ and $(u,t)\in\cl{W}$.
For every $v\in K_t$ there exists a sequence $(v_k)$ 
strongly convergent to $v$ in $X$ with $v_k\in K_{t_k}$.
From
\[
\langle H_{t_k}(u_k),v_k-u_k\rangle\geq \langle\varphi,v_k-u_k\rangle
\]
it follows
\[
\langle H_{t}(u),v-u\rangle\geq \langle\varphi,v-u\rangle
\]
so that the set introduced in assertion~$(a)$ is compact.
\par
Assume now that the problem~\eqref{eq:hvi}
has no solution $(u,t)\in \partial_{X\times[0,1]}\,W$.
It is enough to prove that 
$\left\{t\mapsto \mathrm{deg}((H_{t},K_t),W_t,\varphi)\right\}$
is locally constant.
\par
Suppose first that $K_t\neq\emptyset$ for any $t\in[0,1]$.
By Michael selection theorem (see
e.g.~\cite[Theorem~1.11.1]{aubin_cellina1984})
there exists a continuous map 
$\gamma:[0,1]\rightarrow X$
such that $\gamma(t)\in K_t$ for any $t\in[0,1]$.
\par
If we set
\begin{alignat*}{3}
&\widehat{W}&&=
\left\{(u-\gamma(t),t):\,\,(u,t)\in W\right\}\,,\\
&\widehat{K}_t &&=
\left\{u-\gamma(t):\,\,u\in K_t\right\}\,,\\
&\widehat{H}_t(u) &&= H_t(\gamma(t)+u)\,,
\end{alignat*}
then $\widehat{W}$, $\widehat{K}_t$ and $\widehat{H}$
satisfy the same assumptions and
\[
\mathrm{deg}((\widehat{H}_t,\widehat{K}_t),\widehat{W}_t,\varphi) 
= \mathrm{deg}((H_t,K_t),W_t,\varphi)
\]
by Proposition~\ref{prop:traslX}.
Moreover $0\in \widehat{K}_t$ for any $t\in[0,1]$.
Therefore we may assume, without loss of generality,
that $0\in K_t$ for any $t\in[0,1]$.
\par
Given $t\in[0,1]$, there exist a bounded and open subset
$U$ of $X$ and $\delta>0$ such that
\[
U\times ([t-\delta,t+\delta]\cap[0,1]) \subseteq W
\]
and such that~\eqref{eq:hvi} has no solution $(u,\tau)$ in
\[
W \setminus (U\times [t-\delta,t+\delta])
\]
with $t-\delta\leq\tau\leq t+\delta$.
From Theorem~\ref{thm:add-exc} we infer that
\[
\mathrm{deg}((H_{\tau},K_\tau),W_\tau,\varphi) =
\mathrm{deg}((H_{\tau},K_\tau),U,\varphi)
\qquad\text{for any $\tau\in[t-\delta,t+\delta]$}\,.
\]
From~\cite[Theorem~4.53 
and~Proposition~4.61]{motreanu_motreanu_papageorgiou2014}
we deduce that
$\left\{\tau\mapsto 
\mathrm{deg}((H_{\tau},K_\tau),U,\varphi)\right\}$
is constant on $[t-\delta,t+\delta]$.
\par
In general, given $t\in[0,1]$, let us distinguish the
cases $K_t\neq \emptyset$ and $K_t= \emptyset$.
\par
If $K_t\neq \emptyset$, 
by the Mosco-convergence there exists $\delta>0$ such that
$K_\tau\neq\emptyset$ for any $\tau\in[t-\delta,t+\delta]$.
By the previous step we infer that 
$\left\{\tau\mapsto 
\mathrm{deg}((H_{\tau},K_\tau),W_\tau,\varphi)\right\}$
is constant on $[t-\delta,t+\delta]$.
\par
If $K_t= \emptyset$, from 
Theorem~\ref{thm:existencedegree} we infer that
$\mathrm{deg}((H_{t},K_t),W_t,\varphi)=0$.
Assume, for a contradiction, that there exists a 
sequence $(t_k)$ convergent to $t$ with
$\mathrm{deg}((H_{t_k},K_{t_k}),W_{t_k},\varphi)\neq 0$ 
for any $k\in\N$.
Again from Theorem~\ref{thm:existencedegree} we infer that
the problem~\eqref{eq:hvi}
has a solution $(u_k,t_k)\in \cl{W}$, in particular
$u_k\in K_{t_k}$, for any $k\in\N$.
Up to a subsequence, $(u_k)$ is weakly convegent to some
$u$, whence $u\in K_t$ by the Mosco-convergence, and
a contradiction follows.
\end{proof}
\par\smallskip
Now let $\Omega$ be a bounded and open subset of $\R^n$, 
let $T$ be a metrizable topological space and let
\begin{align*}
& a:\Omega\times
\left(\R\times\R^n\times T\right)\rightarrow \R^n\,,\\
\noalign{\medskip}
& b:\Omega\times
\left(\R\times\R^n\times T\right)\rightarrow \R
\end{align*}
be two Carath\'eodory functions.
We will denote by $\|~\|_p$ the usual norm in $L^p$ and
write $a_t(x,s,\xi)$, $b_t(x,s,\xi)$ instead of
$a(x,(s,\xi,t))$, $b(x,(s,\xi,t))$.
\par
In this section, we assume that $a_t$ and $b_t$ satisfy the 
\emph{controllable growth conditions} in the sense
of~\cite{giaquinta1983}, uniformly with respect to $t$.
In a simplified form enough for our purposes, this means 
that:
\emph{
\begin{enumerate}[align=parleft]
\item[\mylabel{uc}{\uc}]
there exist $p\in]1,\infty[$, $\alpha^{(0)}\in L^1(\Omega)$,
$\alpha^{(1)}\in L^{p'}(\Omega)$, $\beta > 0$ 
and $\nu>0$ such that
\begin{alignat*}{3}
& \left|a_t(x,s,\xi)\right| &&\leq
\alpha^{(1)}(x) + \beta\,|s|^{p-1} 
+ \beta\,|\xi|^{p-1}\,,\\
\noalign{\medskip}
& \left|b_t(x,s,\xi)\right| &&\leq
\alpha^{(1)}(x) + \beta\,|s|^{p - 1} + 
\beta\,|\xi|^{p-1}\,,\\
\noalign{\medskip}
&a_t(x,s,\xi)\cdot \xi &&\geq
\nu |\xi|^p - \alpha^{(0)}(x) - \beta\,|s|^{p}\,,
\end{alignat*}
for a.e. $x\in\Omega$ and every $s\in\R$, $\xi\in\R^n$, 
$t\in T$; such a $p$ is clearly unique.
\end{enumerate}
}
It follows
\[
\left\{
\begin{array}{l}
a_t(x,u,\nabla u) \in L^{p'}(\Omega) \\
\noalign{\medskip}
b_t(x,u,\nabla u) \in L^{p'}(\Omega)
\subseteq W^{-1,p'}(\Omega)
\end{array}
\right.
\qquad\text{for any $t\in T$ and $u\in W^{1,p}_0(\Omega)$}
\]
and the map
$H:W^{1,p}_0(\Omega)\times T\rightarrow W^{-1,p'}(\Omega)$
defined by
\[
H_t(u) = - \mathrm{div}\left[a_t(x,u,\nabla u)\right]
+ b_t(x,u,\nabla u) 
\]
is continuous and bounded on $B\times T$, whenever $B$ is
bounded in $W^{1,p}_0(\Omega)$.
\begin{thm}
\label{thm:controllable}
Assume~\ref{uc} and also that:
\begin{enumerate}
\item[\mylabel{um}{\um}]
we have
\[
\left[a_t(x,s,\xi)-a_t(x,s,\hat{\xi})\right]
\cdot(\xi-\hat{\xi}) > 0
\]
for a.e. $x\in\Omega$ and every $s\in\R$, $\xi,\hat{\xi}\in\R^n$,
$t\in T$, with $\xi\neq\hat{\xi}$.
\end{enumerate}
\indent
Then
$H:W^{1,p}_0(\Omega)\times T\rightarrow W^{-1,p'}(\Omega)$
is of class~$(S)_+$.
\end{thm}
\begin{proof}
See e.g.~\cite[Theorem~1.2.1]{skrypnik1994}.
\end{proof}
%

%--------------------------------------------------------------------

\section{Quasilinear elliptic variational inequalities with
natural growth conditions}
\label{sect:qevin}
Again, let $\Omega$ be a bounded and open subset of $\R^n$ and 
let now
\begin{align*}
& a:\Omega\times
\left(\R\times\R^n\right)\rightarrow \R^n\,,\\
\noalign{\medskip}
& b:\Omega\times
\left(\R\times\R^n\right)\rightarrow \R
\end{align*}
be two Carath\'eodory functions.
In this paper we are interested in the case in which
$a$ and $b$ satisfy the \emph{natural growth conditions}
in the sense of~\cite{giaquinta1983}.
More precisely, we assume that:
\emph{
\begin{enumerate}
\item[\mylabel{hni}{\hni}]
there exist $p\in]1,\infty[$ and, for every $R>0$,
$\alpha^{(0)}_R\in L^1(\Omega)$,
$\alpha^{(1)}_R\in L^{p'}(\Omega)$, $\beta_R > 0$ 
and $\nu_R>0$ such that
\begin{align*}
& |a(x,s,\xi)| \leq
\alpha^{(1)}_R(x) + \beta_R\,|\xi|^{p-1}\,,\\
\noalign{\medskip}
& |b(x,s,\xi)| \leq
\alpha^{(0)}_R(x) + \beta_R\,|\xi|^{p}\,,\\
\noalign{\medskip}
& a(x,s,\xi)\cdot\xi \geq
\nu_R\,|\xi|^{p} - \alpha^{(0)}_R(x)\,,\\
\end{align*}
for a.e. $x\in\Omega$ and every $s\in\R$, $\xi\in\R^n$
with $|s|\leq R$;
such a $p$ is clearly unique;
\item[\mylabel{hnii}{\hnii}]
we have
\[
\left[a(x,s,\xi)-a(x,s,\hat{\xi})\right]
\cdot(\xi-\hat{\xi}) > 0
\]
for a.e. $x\in\Omega$ and every $s\in\R$, $\xi,\hat{\xi}\in\R^n$
with $\xi\neq\hat{\xi}$.
\end{enumerate}
}
Then we can define a map
\[
F:W^{1,p}_0(\Omega)\cap L^{\infty}(\Omega)
\rightarrow W^{-1,p'}(\Omega) + L^1(\Omega)
\]
by
\[
F(u) = - \mathrm{div}\left[a(x,u,\nabla u)\right]
+ b(x,u,\nabla u) \,.
\]
\begin{rem}
Assume that $\hat{a}$ and $\hat{b}$ also 
satisfy~\ref{hni} and~\ref{hnii}.
An easy density argument shows that, if
\begin{multline*}
\int_{\Omega}
\bigl[a(x,u,\nabla u)\cdot\nabla v +
b(x,u,\nabla u)\,v\bigr]\,dx 
=
\int_{\Omega}
\bigl[\hat{a}(x,u,\nabla u)\cdot\nabla v +
\hat{b}(x,u,\nabla u)\,v\bigr]\,dx \\
\text{for any $u\in W^{1,p}_0(\Omega)\cap L^\infty(\Omega)$
and any $v\in C^\infty_c(\Omega)$}\,,
\end{multline*}
then
\begin{multline*}
\int_{\Omega}
\bigl[a(x,u,\nabla u)\cdot\nabla v +
b(x,u,\nabla u)\,v\bigr]\,dx 
=
\int_{\Omega}
\bigl[\hat{a}(x,u,\nabla u)\cdot\nabla v +
\hat{b}(x,u,\nabla u)\,v\bigr]\,dx \\
\text{for any $u\in W^{1,p}_0(\Omega)\cap L^\infty(\Omega)$
and any $v\in W^{1,p}_0(\Omega)\cap L^\infty(\Omega)$}\,.
\end{multline*}
\end{rem}
Consider also a $p$-quasi upper semicontinuous 
function $\underline{u}:\Omega\rightarrow\Rb$
and a $p$-quasi lower semicontinuous 
function $\overline{u}:\Omega\rightarrow\Rb$,
and set
\[
K=\left\{u\in W^{1,p}_0(\Omega)\cap L^{\infty}(\Omega):
\,\,\text{$\underline{u}\leq \tilde{u}\leq\overline{u}$
\,\,$p$-q.e. in $\Omega$}\right\}\,,
\]
where $\tilde{u}$ is any $p$-quasi continuous representative
of $u$ (see e.g.~\cite{dalmaso1983}).
\par
We aim to consider the solutions of the variational inequality
\begin{equation}
\label{eq:qevi}
\tag{$VI$}
\begin{cases}
u\in K\,,\\
\noalign{\medskip}
\displaystyle{
\int_{\Omega}
\bigl[a(x,u,\nabla u)\cdot\nabla (v-u) +
b(x,u,\nabla u)\,(v-u)\bigr]\,dx \geq 0}\\
\noalign{\medskip}
\qquad\qquad\qquad\qquad\qquad\qquad\qquad\qquad\qquad\qquad
\text{for every $v\in K$}\,.
\end{cases}
\end{equation}
We denote by $Z^{tot}(F,K)$ the set of solutions $u$
of~\eqref{eq:qevi}.
We will simply write $Z^{tot}$, if no confusion can arise.
\par
For every $u\in K$, we also set
\begin{multline*}
T_uK = \biggl\{v\in W^{1,p}_0(\Omega)\cap L^{\infty}(\Omega):\,\,
\text{$\tilde{v}\geq 0$\,\,$p$-q.e. in 
$\{\tilde{u}=\underline{u}\}$
and
} \\ \text{
$\tilde{v}\leq 0$\,\,$p$-q.e. in 
$\{\tilde{u}=\overline{u}\}$}\biggr\}\,.
\end{multline*}
\begin{prop}
\label{prop:TK}
A function $u\in K$ satisfies~\eqref{eq:qevi} if and  only if
\[
\int_{\Omega}
\bigl[a(x,u,\nabla u)\cdot\nabla v +
b(x,u,\nabla u)\,v\bigr]\,dx \geq 0
\qquad\text{for every $v\in T_uK$}\,.
\]
\end{prop}
\begin{proof}
Assume that $u$ is a solution of~\eqref{eq:qevi} and
let $v\in T_uK$ with $v\leq 0$ a.e. in $\Omega$.
Since $\max\{k(\underline{u}-\tilde{u}),\tilde{v}\}$ is a 
nonincreasing sequence of nonpositive $p$-quasi upper 
semicontinuous functions converging to $\tilde{v}$
$p$-q.e. in $\Omega$,
by~\cite[Lemma~1.6]{dalmaso1983} there exists a
sequence $(v_k)$ in $W^{1,p}_0(\Omega)$
converging to $v$ in $W^{1,p}_0(\Omega)$ with
$\tilde{v}_k \geq \max\{k(\underline{u}-\tilde{u}),\tilde{v}\}$ 
$p$-q.e. in $\Omega$.
Without loss of generality, we may assume that 
$\tilde{v}_k\leq 0$ $p$-q.e. in $\Omega$.
\par
Then it follows that $u+\frac{1}{k}\,v_k \in K$, whence
\[
\int_{\Omega}
\bigl[a(x,u,\nabla u)\cdot\nabla v_k +
b(x,u,\nabla u)\,v_k\bigr]\,dx \geq 0\,.
\]
Going to the limit as $k\to\infty$, we get
\[
\int_{\Omega}
\bigl[a(x,u,\nabla u)\cdot\nabla v +
b(x,u,\nabla u)\,v\bigr]\,dx \geq 0\,.
\]
If $v\in T_uK$ with $v\geq 0$ a.e. in $\Omega$,
the argument is similar.
Since every $v\in T_uK$ can be written as
$v=v^+ - v^-$ with $v^+, - v^-\in T_uK$,
it follows
\[
\int_{\Omega}
\bigl[a(x,u,\nabla u)\cdot\nabla v +
b(x,u,\nabla u)\,v\bigr]\,dx \geq 0
\qquad\text{for every $v\in T_uK$}\,.
\]
Since $K\subseteq u + T_uK$, the converse is obvious.
\end{proof}
We are also interested in the invariance of the problem
with respect to suitable transformations.
\par
Let us denote by $\Phi$ the set of increasing 
$C^2$-diffeomorphisms
$\varphi:\R\rightarrow\R$ such that $\varphi(0)=0$ and by
$\Theta$ the set of $C^1$-functions
$\vartheta:\R\rightarrow]0,+\infty[$.
\par
For any $\varphi\in\Phi$ and $\vartheta\in \Theta$,
we define
\[
F^\varphi\,,\,\,F_\vartheta:
W^{1,p}_0(\Omega)\cap L^{\infty}(\Omega)
\rightarrow W^{-1,p'}(\Omega) + L^1(\Omega)
\]
by
\[
F^\varphi(u) = F(\varphi(u))\,,\,\,
F_\vartheta(u) = \vartheta(u)\,F(u)\,.
\]
If we define the Carath\'eodory functions
\[
a^{\varphi}, a_\vartheta:
\Omega\times\left(\mathbb{R}\times\mathbb{R}^{n}
\right)\rightarrow\mathbb{R}^{n}\,,\qquad
b^{\varphi}, b_\vartheta:
\Omega\times\left(\mathbb{R}\times\mathbb{R}^{n}
\right)\rightarrow\mathbb{R}
\]
by
\begin{alignat*}{3}
&a^{\varphi}(x,s,\xi)&&=
\varphi'(s)\,a(x,\varphi(s),\varphi'(s)\xi)
\,,\\
&b^{\varphi}(x,s,\xi)&&=
\varphi''(s)\,a(x,\varphi(s),\varphi'(s)\xi)\cdot\xi 
+\varphi'(s)\,b(x,\varphi(s),\varphi'(s)\xi)\,,\\
&a_{\vartheta}(x,s,\xi)&&=
\vartheta(s) \, a(x,s,\xi)\,,\\
&b_{\vartheta}(x,s,\xi)&&=
\vartheta'(s) \, a(x,s,\xi)\cdot\xi 
+\vartheta(s) \, b(x,s,\xi)\,,
\end{alignat*}
it easily follows that
\begin{alignat*}{3}
&F^\varphi(u) &&= 
- \mathrm{div}\left[a^\varphi(x,u,\nabla u)\right]
+ b^\varphi(x,u,\nabla u) \,,\\
&F_\vartheta(u) &&= 
- \mathrm{div}\left[a_\vartheta(x,u,\nabla u)\right]
+ b_\vartheta(x,u,\nabla u) \,.
\end{alignat*}
We also set $u^{\varphi} = \varphi^{-1}(u)$ and,
given a set $E$ of real valued functions, 
$E^{\varphi} = \left\{u^{\varphi}:\,\,u\in E\right\}$.
\par
It is easily seen that 
\[
(a^\varphi)^\psi = a^{\varphi\circ\psi}\,,\qquad
(a_\vartheta)_\varrho = a_{\vartheta\varrho}\,,\qquad
(u^\varphi)^\psi = u^{\varphi\circ\psi}\,,
\]
for every $\varphi, \psi\in\Phi$ and $\vartheta, \varrho\in\Theta$.
\par
We also say that $(a,b)$ is of 
\emph{Euler-Lagrange type}, if there exists a function
\[
L:\Omega\times\mathbb{R}\times\mathbb{R}^{n}\rightarrow\mathbb{R}
\]
such that
\begin{enumerate}
\item[$(a)$]
$\left\{x\mapsto L(x,s,\xi)\right\}$ is measurable
for every $(s,\xi)\in\R\times\R^n$;
\item[$(b)$]
$\left\{(s,\xi)\mapsto L(x,s,\xi)\right\}$ is of class $C^1$
for a.e. $x\in\Omega$;
\item[$(c)$]
we have
\[
a(x,s,\xi) = \nabla_\xi L(x,s,\xi)\,,\qquad
b(x,s,\xi) = D_s L(x,s,\xi)\,,
\]
for a.e. $x\in\Omega$ and every $(s,\xi)\in\R\times\R^n$.
\end{enumerate}
\indent
Taking into account Proposition~\ref{prop:TK},
the next two results are easy to prove.
\begin{prop}
\label{prop:phi}
For every $\varphi\in\Phi$, the following facts hold:
\begin{enumerate}
\item[$(a)$]
the functions $a^\varphi, b^\varphi$ 
satisfy~\ref{hni} and~\ref{hnii} with the same $p$;
\item[$(b)$]
we have
\[
K^{\varphi} = 
\left\{u\in W^{1,p}_0(\Omega)\cap L^\infty(\Omega):\,\,
\text{$\underline{u}^\varphi\leq \tilde{u}\leq
\overline{u}^\varphi$
\,\,$p$-q.e. in $\Omega$}\right\}
\]
and $\underline{u}^\varphi:\Omega\rightarrow\Rb$
is $p$-quasi upper semicontinuous, while
$\overline{u}^\varphi:\Omega\rightarrow\Rb$
is $p$-quasi lower semicontinuous
(here we agree that $\varphi(-\infty)=-\infty$ and 
$\varphi(+\infty)=+\infty$);
\item[$(c)$]
$u$ is a solution of~\eqref{eq:qevi}
if and only if $u^\varphi$ is a 
solution of the corresponding variational inequality with
$a$, $b$ and $K$ replaced by $a^\varphi$, $b^\varphi$ and 
$K^\varphi$, respectively;
\item[$(d)$]
the pair $(a,b)$ is of Euler-Lagrange type if and only if the pair 
$(a^\varphi,b^\varphi)$ does the same;
moreover, if $L$ is associated with $(a,b)$, then
\[
L^\varphi(x,s,\xi) =
L(x,\varphi(s),\varphi'(s)\xi)
\]
is associated with $(a^\varphi,b^\varphi)$.
\end{enumerate}
\end{prop}
\begin{prop}
\label{prop:theta}
For every $\vartheta\in\Theta$, the following facts hold:
\begin{enumerate}
\item[$(a)$]
the functions 
$a_\vartheta, b_\vartheta$ 
satisfy~\ref{hni} and~\ref{hnii} with the same $p$;
\item[$(b)$]
$u$ is a solution of~\eqref{eq:qevi}
if and only if the same $u$ is a 
solution of the corresponding variational inequality with
$a$ and $b$ replaced by $a_\vartheta$ and $b_\vartheta$,
respectively.
\end{enumerate}
\end{prop}
\begin{rem}
Let $\varphi\in\Phi$ with $\varphi'$ nonconstant, 
let $w\in L^2(\Omega)$ and let
$a(x,s,\xi)=\varphi'(s)\xi$, $b(x,s,\xi) = -w(x)$.
\par
Then $(a,b)$ is not of Euler-Lagrange type.
However, if we take $\vartheta(s)=\varphi'(s)$, then
$(a_\vartheta,b_\vartheta)$, which is given by
\[
a_{\vartheta}\left(x,s,\xi\right)=
[\varphi'(s)]^2 \xi\,,\qquad
b_{\vartheta}\left(x,s,\xi\right)=
\varphi'(s) \varphi''(s) |\xi|^2 
- w(x) \varphi'(s) \,,
\]
is of Euler-Lagrange type with
\[
L(x,s,\xi) = \frac{1}{2}\,[\varphi'(s)]^2 |\xi|^2 
- w(x)\varphi(s)\,.
\]
Therefore the property of being of Euler-Lagrange
type is not invariant under the transformation 
induced by $\vartheta$, which plays in fact the role
of ``integrating factor''.
By the way, if then we take $\psi=\varphi^{-1}$, we get
\[
(a_\vartheta)^\psi(x,s,\xi) = \xi\,,\qquad
(b_\vartheta)^\psi(x,s,\xi) = - w(x)\,,
\]
which are simply related to 
\[
L(x,s,\xi) = \frac{1}{2}\,|\xi|^2 - w(x) s\,.
\]
\end{rem}
\begin{prop}
\label{prop:signgen}
For every $R>0$ there exist $\vartheta_1, \vartheta_2\in\Theta$
and $c>0$, depending only on $R$, $\beta_R$ and $\nu_R$, such that
\begin{alignat*}{3}
&b_{\vartheta_1}(x,s,\xi) &&\leq
c\, \alpha_R^{(0)}(x) \,,\\
&b_{\vartheta_2}(x,s,\xi) &&\geq
- c\, \alpha_R^{(0)}(x)  \,,
\end{alignat*}
for a.e. $x\in\Omega$ and every $s\in\R$, $\xi\in\R^n$
with $|s|\leq R$.
\end{prop}
\begin{proof}
If we set $\vartheta(s) = \exp(\gamma s)$ with 
$\gamma\, \nu_R\geq\beta_R$, we have
\[
\begin{split}
b_\vartheta(x,s,\xi) &= 
\left[\gamma\,a(x,s,\xi)\cdot\xi + b(x,s,\xi)
\right]\exp(\gamma s)\\
&\geq
\left[\gamma\,\nu_R\,|\xi|^{p} - \gamma\,\alpha^{(0)}_R(x)
- \alpha^{(0)}_R(x) - \beta_R\,|\xi|^{p}
\right]\exp(\gamma s) \\
&\geq
- (\gamma+1) \exp(\gamma R)\,\alpha^{(0)}_R(x)\,,
\end{split}
\]
whence the existence of $\vartheta_2$.
The existence of $\vartheta_1$ can be proved in a similar way.
\end{proof}

%--------------------------------------------------------------------

\section{Quasilinear elliptic variational inequalities with
natural growth conditions depending on a parameter}
\label{sect:qevinp}
Again, let $\Omega$ be a bounded and open subset of $\R^n$
and let now $T$ be a metrizable topological space and
\begin{align*}
& a:\Omega\times
\left(\R\times\R^n\times T\right)\rightarrow \R^n\,,\\
\noalign{\medskip}
& b:\Omega\times
\left(\R\times\R^n\times T\right)\rightarrow \R
\end{align*}
be two Carath\'eodory functions 
satisfying~\ref{hni} and~\ref{hnii} uniformly
with respect to $t\in T$.
\par
More precisely, we assume that $a_t$ and $b_t$
satisfy~\ref{um} and:
\emph{
\begin{enumerate}
\item[\mylabel{uni}{\uni}]
there exist $p\in]1,\infty[$ and, for every $R>0$,
$\alpha^{(0)}_R\in L^1(\Omega)$,
$\alpha^{(1)}_R\in L^{p'}(\Omega)$, $\beta_R > 0$ 
and $\nu_R>0$ such that
\begin{align*}
& |a_t(x,s,\xi)| \leq
\alpha^{(1)}_R(x) + \beta_R\,|\xi|^{p-1}\,,\\
\noalign{\medskip}
& |b_t(x,s,\xi)| \leq
\alpha^{(0)}_R(x) + \beta_R\,|\xi|^{p}\,,\\
\noalign{\medskip}
& a_t(x,s,\xi)\cdot\xi \geq
\nu_R\,|\xi|^{p} - \alpha^{(0)}_R(x)\,,
\end{align*}
for a.e. $x\in\Omega$ and every $s\in\R$, $\xi\in\R^n$
and $t\in T$ with $|s|\leq R$;
again, such a $p$ is clearly unique.
\end{enumerate}
}
Then we can define a map
\[
H:\left[W^{1,p}_0(\Omega)\cap L^{\infty}(\Omega)\right]\times T
\rightarrow \left[W^{-1,p'}(\Omega) + L^1(\Omega)\right]
\]
by
\[
H_t(u) = - \mathrm{div}\left[a_t(x,u,\nabla u)\right]
+ b_t(x,u,\nabla u) \,.
\]
Consider also, for each $t\in T$, a $p$-quasi upper 
semicontinuous function $\underline{u}_t:\Omega\rightarrow\Rb$
and a $p$-quasi lower semicontinuous function 
$\overline{u}_t:\Omega\rightarrow\Rb$, set
\[
K_t=\left\{u\in W^{1,p}_0(\Omega)\cap L^{\infty}(\Omega):
\,\,\text{$\underline{u}_t\leq \tilde{u}\leq\overline{u}_t$
\,\,$p$-q.e. in $\Omega$}\right\}
\]
and assume the following form of continuity related to the 
Mosco-convergence:
\emph{
\begin{enumerate}
\item[\mylabel{mc}{\mc}]
for every sequence $(t_k)$ convergent to $t$ in $T$, 
the following facts hold:
\begin{itemize}
\item
if $(u_k)$ is a sequence weakly convergent to $u$ in 
$W^{1,p}_0(\Omega)$, with $u_k\in K_{t_k}$ for any $k\in\N$ 
and $(u_k)$ bounded in~$L^\infty(\Omega)$, then $u\in K_t$;
\item
for every $u\in K_t$ there exist $\overline{k}\in\N$ and a 
sequence $(u_k)$ in $W^{1,p}_0(\Omega)\cap L^{\infty}(\Omega)$
which is bounded in $L^\infty(\Omega)$ and
strongly convergent to $u$ in $W^{1,p}_0(\Omega)$, 
with $u_k\in K_{t_k}$ for any $k\geq \overline{k}$.
\end{itemize}
\end{enumerate}
}
Then consider the parametric variational inequality
\begin{equation}
\label{eq:qevim}
\tag{$PVI$}
\begin{cases}
(u,t)\in [W^{1,p}_0(\Omega)\cap L^{\infty}(\Omega)]\times T\,,\\
\noalign{\medskip}
u\in K_t\,, \\
\noalign{\medskip}
\displaystyle{
\int_{\Omega}
\bigl[a_t(x,u,\nabla u)\cdot\nabla (v-u) +
b_t(x,u,\nabla u)\,(v-u)\bigr]\,dx \geq 0}\\
\noalign{\medskip}
\qquad\qquad\qquad\qquad\qquad\qquad\qquad\qquad\qquad\qquad
\text{for every $v\in K_t$}\,.
\end{cases}
\end{equation}
\begin{thm}
\label{thm:proper}
Let $(u_k,t_k)$ be a sequence of solutions of~\eqref{eq:qevim}
with $(u_k)$ bounded in $L^{\infty}(\Omega)$ and $(t_k)$ 
convergent to some $t$ in $T$ with $K_t\neq\emptyset$.
\par
Then $(u_k)$ admits a subsequence strongly convergent in
$W^{1,p}_0(\Omega)$ to some $u$ and $(u,t)$ is a
solution of~\eqref{eq:qevim}.
\end{thm}
\begin{proof}
Let $w\in K_t$ and let $(w_k)$ be a sequence strongly convergent 
to $w$ in $W^{1,p}_0(\Omega)$, with $w_k\in K_{t_k}$ for any $k\in\N$ 
and $(w_k)$ bounded in $L^\infty(\Omega)$.
If we set
\begin{gather*}
\widehat{T} = \N\cup\{\infty\}\,,\\
\hat{a}_k(x,s,\xi) = a_{t_k}(x,w_k(x)+s,\nabla w_k(x)+\xi) \,,\\
\hat{a}_\infty(x,s,\xi) = a_{t}(x,w(x)+s,\nabla w(x)+\xi) \,,\\
\hat{b}_k(x,s,\xi) = b_{t_k}(x,w_k(x)+s,\nabla w_k(x)+\xi) \,,\\
\hat{b}_\infty(x,s,\xi) = b_{t}(x,w(x)+s,\nabla w(x)+\xi) \,,\\
\hat{\underline{u}}_k = \underline{u}_{t_k} - w_k\,,\qquad
\hat{\underline{u}}_\infty = \underline{u}_{t} - w\,,\\
\hat{\overline{u}}_k = \overline{u}_{t_k} - w_k\,,\qquad
\hat{\overline{u}}_\infty = \overline{u}_{t} - w\,,\\
\hat{u}_k = u_k - w_k\,,
\end{gather*}
and define $\widehat{K}_k$, $\widehat{K}_\infty$ accordingly, 
it is easily seen that all the assumptions are still
satisfied and now $0\in \widehat{K}_{k}$.
Therefore, we may assume without loss of generality that
$0\in K_{t_k}$ for any $k\in\N$.
\par
Let $R>0$ be such that $\|u_k\|_\infty\leq R$ for any $k\in\N$.
We claim that $(u_k^+)$ is bounded in~$W^{1,p}_0(\Omega)$.
Actually, by Propositions~\ref{prop:theta} 
and~\ref{prop:signgen}, we may assume,
without loss of generality, that
\[
b_t(x,s,\xi) \geq - c \,\alpha_R^{(0)}(x) 
\qquad\text{whenever $|s|\leq R$}\,.
\]
Then the choice $v=-u_k^-$ in~\eqref{eq:qevim} yields
\begin{multline*}
0 \geq 
\int_{\Omega}
\bigl[{a}_{t_k}(x,u_k^+,\nabla u_k^+)\cdot\nabla u_k^+ +
{b}_{t_k}(x,u_k^+,\nabla u_k^+)\,u_k^+\bigr]\,dx \\
\geq
\nu_R \int_\Omega |\nabla u_k^+|^p\,dx 
- \int_\Omega \alpha_R^{(0)}\,dx 
- c \int_\Omega \alpha_R^{(0)}\, u_k^+\,dx\,,
\end{multline*}
which implies that $(u_k^+)$ is bounded in $W^{1,p}_0(\Omega)$.
\par
In a similar way one finds that 
$(u_k^-)$ is bounded in $W^{1,p}_0(\Omega)$,
so that $(u_k)$ is weakly convergent, up to a subsequence,
to some $u$ in $W^{1,p}_0(\Omega)$ and $u\in K_t$.
Let $(z_k)$ be a sequence strongly convergent to $u$ in 
$W^{1,p}_0(\Omega)$, with $z_k\in K_{t_k}$ for any $k\in\N$ and 
$(z_k)$ bounded in $L^\infty(\Omega)$.
\par
Let $\psi:\R\rightarrow[0,1]$ be a continuous function
such that $\psi(s)=1$ for $|s|\leq R$ and $\psi(s) = 0$
for $|s|\geq R+1$.
Then, let 
\begin{alignat*}{3}
&\check{a}_t(x,s,\xi) &&= 
\psi(s)\,a_t(x,s,\xi) + (1-\psi(s))\,|\xi|^{p-2}\xi\,,\\
&\check{b}_t(x,s,\xi) &&= \psi(s)\,b_t(x,s,\xi) \,.
\end{alignat*}
It is easily seen that each $(u_k,t_k)$ is also a solution
of~\eqref{eq:qevim} with $a_t$ and $b_t$ replaced by
$\check{a}_t$ and~$\check{b}_t$.
Moreover, $\check{a}_t$ and $\check{b}_t$ satisfy both
the assumptions~\ref{uni}, \ref{um} and the assumptions
of~\cite[Theorem~4.2]{almi_degiovanni2013}.
In particular, there exist $\alpha^{(0)}\in L^1(\Omega)$,
$\beta > 0$ and $\nu>0$ such that
\[
\check{a}_t(x,s,\xi)\cdot\xi \geq
\nu\,|\xi|^{p} - \alpha^{(0)}(x)\,,
\qquad
|\check{b}_t(x,s,\xi)| \leq
\alpha^{(0)}(x) + \beta\,|\xi|^{p}
\]
for a.e. $x\in\Omega$ and every $s\in\R$, $\xi\in\R^n$ and $t\in T$.
\par
Now, if $p<n$, the proof 
of~\cite[Theorem~4.2]{almi_degiovanni2013}
can be repeated in a simplified
form, as $(u_k)$ is bounded in $L^\infty(\Omega)$.
We have only to observe that, if
$\varphi:\R\rightarrow\R$ is the solution of
\[
\left\{
\begin{array}{l}
\varphi'(s) = 1 + \dfrac{\beta}{\nu}\,|\varphi(s)|\,,\\
\noalign{\medskip}
\varphi(0)=0\,,
\end{array}
\right.
\]
then there exists $\tau>0$ such that
\[
0\leq \tau s\varphi(s) \leq s^2
\qquad\text{whenever $|s|\leq R+\sup_k\,\|z_k\|_\infty$}\,, 
\]
so that
\[
u_k - \tau\,\varphi(u_k-z_k)\in K_{t_k}
\qquad\text{for any $k\in\N$}\,.
\]
It follows
\[
\int_{\Omega}
\bigl[\check{a}_{t_k}(x,u_k,\nabla u_k)\cdot
\nabla (\varphi(u_k-z_k)) +
\check{b}_{t_k}(x,u_k,\nabla u_k)\,\varphi(u_k-z_k)\bigr]\,dx 
\leq 0
\]
and now the proof of~\cite[Theorem~4.2]{almi_degiovanni2013}
can be repeated with minor modifications,
showing that $(u_k)$ is strongly convergent to $u$
in $W^{1,p}_0(\Omega)$.
If $p\geq n$, the argument is similar and simpler.
\par
It is easily seen that $(u,t)$ is a
solution of~\eqref{eq:qevim}.
\end{proof}
\begin{rem}
In the previous theorem the assumption $K_t\neq\emptyset$
is crucial to ensure that $(u_k)$ is bounded in
$W^{1,p}_0(\Omega)$.
\par
Consider $\Omega=]0,1[$, $a_t(x,s,\xi)=\xi$,
$b_t(x,s,\xi)=0$, $\underline{u}_t=z-t$ and
$\overline{u}_t=z+t$, where 
$z\in C_c(]0,1[)\setminus W^{1,2}_0(]0,1[)$.
\par
If $t_k\to 0$ with $t_k>0$ and $(u_k,t_k)$ are
the solutions of~\eqref{eq:qevim}, then 
the sequence $(u_k)$ is unbounded in $W^{1,2}_0(\Omega)$.
\end{rem}

%--------------------------------------------------------------------

\section{Topological degree for quasilinear elliptic variational
inequalities with natural growth conditions}
\label{sect:tdvi}
Consider again the setting of Section~\ref{sect:qevin}.
Throughout this section, we also assume that:
\emph{
\begin{enumerate}
\item[\mylabel{hniii}{\hniii}]
the functions $\underline{u}$ and $\overline{u}$
are bounded.
\end{enumerate}
}
It follows that $Z^{tot}$ is automatically bounded in
$L^\infty(\Omega)$.
\begin{thm}
\label{thm:compact}
The set $Z^{tot}$ is (strongly) compact in 
$W^{1,p}_0(\Omega)$ (possibly empty).
\end{thm}
\begin{proof}
If $K=\emptyset$, we have $Z^{tot}=\emptyset$.
Otherwise the assertion follows from Theorem~\ref{thm:proper}.
\end{proof}
\begin{defn}
We denote by $\mathcal{Z}(F,K)$
the family of the subsets $Z$ of $Z^{tot}$
which are both open and closed in $Z^{tot}$ with respect to the 
$W^{1,p}_0(\Omega)$-topology.
We will simply write~$\mathcal{Z}$, if no confusion can arise.
\end{defn}
Fix a continuous function $\psi:\R\rightarrow[0,1]$ such that
\[
\psi(s)=1 \quad\text{for $s\leq 1$}\,,\qquad
\psi(s)=0 \quad\text{for $s\geq 2$}\,,
\]
then set, for any $t\in [0,1]$ and $s\in\R$,
\[
\Psi_{t}(s)=\psi(t |s|)\,s\,.
\]
If we consider $T=[0,1]$ and define
\begin{alignat*}{3}
& a_t^\psi(x,s,\xi) = \psi(t |s|)\,a(x,s,\xi)
+ [1-\psi(t |s|)]\,|\xi|^{p-2}\xi\,,\\
& b_t^\psi(x,s,\xi) = \Psi_{t}\left(b(x,s,\xi)\right)\,,
\end{alignat*}
it is easily seen that $a_t^\psi, b_t^\psi$ 
satisfy~\ref{uni} and~\ref{um}.
Moreover, for every $\underline{t}\in]0,1[$, they 
satisfy~\ref{uc},
if $t$ is restricted to $[\underline{t},1]$.
In particular, we can define a continuous map
$H^\psi:W^{1,p}_0(\Omega)\times ]0,1]\rightarrow W^{-1,p'}(\Omega)$
by
\[
H_t^\psi(u) = - \mathrm{div}\left[a_t^\psi(x,u,\nabla u)\right]
+ b_t^\psi(x,u,\nabla u)
\]
and, by Theorem~\ref{thm:controllable},
this map is of class~$(S)_+$.
We will simply write $H$, if no confusion can arise.
\begin{prop}
\label{prop:Z}
For every $Z\in\mathcal{Z}$, the following facts hold:
\begin{enumerate}
\item[$(a)$]
there exist a bounded and open subset $U$ of $W^{1,p}_0(\Omega)$
and $\overline{t}\in]0,1]$ such that 
\[
Z = Z^{tot}\cap U = Z^{tot}\cap\overline{U}
\]
and such that the variational inequality~\eqref{eq:qevim} has no 
solution $(u,t)\in \partial U\times[0,\overline{t}]$;
in particular, the degree
$\mathrm{deg}((H_{t}^\psi,K),U,0)$
is defined whenever $t \in]0,\overline{t}]$;
\item[$(b)$]
if $\psi_0,\psi_1:\R\rightarrow[0,1]$ have the same
properties of $\psi$ and $U_0$, $\overline{t}_0$ and 
$U_1$, $\overline{t}_1$ are as in~$(a)$, then
\[
\mathrm{deg}((H_{t}^{\psi_0},K),U_0,0) = 
\mathrm{deg}((H_{\tau}^{\psi_1},K),U_1,0)
\]
for every $t \in]0,\overline{t}_0]$ and
$\tau \in]0,\overline{t}_1]$;
\item[$(c)$]
if $\hat{a}$ and $\hat{b}$ also 
satisfy~\ref{hni}, \ref{hnii} and
\begin{multline*}
~\qquad\qquad
\int_{\Omega}
\bigl[a(x,u,\nabla u)\cdot\nabla v +
b(x,u,\nabla u)\,v\bigr]\,dx \\
=
\int_{\Omega}
\bigl[\hat{a}(x,u,\nabla u)\cdot\nabla v +
\hat{b}(x,u,\nabla u)\,v\bigr]\,dx \\
\text{for any $u\in W^{1,p}_0(\Omega)\cap L^\infty(\Omega)$
and any $v\in C^\infty_c(\Omega)$}\,,
\end{multline*}
then we have
\[
\mathrm{deg}((H_{t}^\psi,K),U,0) = 
\mathrm{deg}((\widehat{H}_{\tau}^\psi,K),\widehat{U},0)
\]
for every $t \in]0,\overline{t}]$ and
$\tau \in]0,\hat{t}]$,
provided that $U$, $\overline{t}$, $\widehat{U}$, $\hat{t}$
are as in~$(a)$ with respect to $a,b$ and $\hat{a},\hat{b}$,
respectively.
\end{enumerate}
\end{prop}
\begin{proof}
By definition of $\mathcal{Z}$, there exists a bounded and open 
subset $U$ of $W^{1,p}_0(\Omega)$ such that 
\[
Z = Z^{tot}\cap U = Z^{tot}\cap\overline{U}\,.
\]
If $(u_k,t_k)$ is a sequence of solutions
of~\eqref{eq:qevim} with $u_k\in \partial U$ and $t_k \to 0$,
then $(u_k)$ is bounded in $L^{\infty}(\Omega)$ by~\ref{hniii}.
By Theorem~\ref{thm:proper}, up to a subsequence $(u_k)$
is convergent to some $u$ in $W^{1,p}_0(\Omega)$ and $u$
is a solution of~\eqref{eq:qevi}.
Then $u\in \partial U$ and a contradiction follows.
Therefore, there exists $\overline{t}\in]0,1]$ such 
that~\eqref{eq:qevim} has no solution 
$(u,t)\in \partial U\times[0,\overline{t}]$
and assertion~$(a)$ is proved.
\par
To prove~$(b)$, define
\[
\psi_{\mu} = (1-\mu)\psi_0+\mu\psi_1
\]
and consider $H_{t}^{\psi_\mu}$.
Arguing as before, we find
$\underline{t}>0$ with $\underline{t}\leq\overline{t}_j$,
$j=0,1$, such that 
\[
\begin{cases}
u\in K\,,\\
\langle H_{t}^{\psi_\mu}(u),v-u\rangle \geq 0
&\qquad\forall v\in K\,,
\end{cases}
\]
has no solution with 
$u\in (\cl{U_0}\setminus U_1)\cup(\cl{U_1}\setminus U_0)$,
$t\in[0,\underline{t}]$ and $\mu\in[0,1]$. 
\par
From Theorem~\ref{thm:add-exc} we infer that
\begin{alignat*}{3}
&\mathrm{deg}((H_{\underline{t}}^{\psi_0},K),U_0,0) &&= 
\mathrm{deg}((H_{\underline{t}}^{\psi_0},K),U_0\cap U_1,0)\,,\\
&\mathrm{deg}((H_{\underline{t}}^{\psi_1},K),U_1,0) &&= 
\mathrm{deg}((H_{\underline{t}}^{\psi_1},K),U_0\cap U_1,0)\,.
\end{alignat*}
On the other hand, we have
\begin{alignat*}{3}
&\mathrm{deg}((H_{\underline{t}}^{\psi_0},K),U_0\cap U_1,0) &&= 
\mathrm{deg}((H_{\underline{t}}^{\psi_1},K),U_0\cap U_1,0) \,,\\
&\mathrm{deg}((H_{\underline{t}}^{\psi_0},K),U_0,0) &&= 
\mathrm{deg}((H_{t}^{\psi_0},K),U_0,0) 
&&\qquad\forall t\in]0,\overline{t}_0]\,,\\
&\mathrm{deg}((H_{\underline{t}}^{\psi_1},K),U_1,0) &&= 
\mathrm{deg}((H_{\tau}^{\psi_1},K),U_1,0) 
&&\qquad\forall \tau\in]0,\overline{t}_1]\,,
\end{alignat*}
by Theorem~\ref{thm:homotopydegree}.
Then assertion~$(b)$ also follows.
\par
The proof of~$(c)$ is quite similar.
\end{proof}
\begin{defn}
\label{defn:degree}
For every $Z\in\mathcal{Z}(F,K)$, we set
\[
\degc((F,K),Z) = 
\mathrm{deg}((H_t,K),U,0)\,,
\]
where $\psi$, $U$, $\overline{t}$ are as in~$(a)$ and
$0<t\leq \overline{t}$.
We will simply write $\degc(Z)$, if no confusion can arise.
\end{defn}
\begin{prop}
\label{prop:consistency}
Assume that $a$ and $b$ satisfy, instead of~\ref{hni},
the more specific controllable growth condition:
\begin{enumerate}
\item[\mylabel{hci}{\hci}]
there exist $p\in]1,\infty[$, $\alpha^{(0)}\in L^{1}(\Omega)$,
$\alpha^{(1)}\in L^{p'}(\Omega)$, $\beta > 0$ and $\nu>0$
such that
\begin{alignat*}{3}
& \left|a(x,s,\xi)\right| \leq
\alpha^{(1)}(x) + \beta\,|s|^{p-1} + \beta\,|\xi|^{p-1}\,,\\
\noalign{\medskip}
& \left|b(x,s,\xi)\right| \leq
\alpha^{(1)}(x) + \beta\,|s|^{p - 1} + 
\beta\,|\xi|^{p-1}\,,\\
\noalign{\medskip}
&a(x,s,\xi)\cdot \xi \geq
\nu |\xi|^p - \alpha^{(0)}(x) - \beta\,|s|^{p}\,,
\end{alignat*}
for a.e. $x\in\Omega$ and every $s\in\R$, $\xi\in\R^n$.
\end{enumerate}
\indent
Then the map $F:W^{1,p}_0(\Omega)\rightarrow W^{-1,p'}(\Omega)$
is continuous, bounded on bounded subsets and of class~$(S)_+$.
Moreover, if $U$ is a bounded and open subset of $W^{1,p}_0(\Omega)$
such that~\eqref{eq:qevi} has no solution $u\in \partial U$ and
\[
Z = \left\{u\in U:
\,\,\text{$u$ is a solution of~\eqref{eq:qevi}}\right\}\,,
\]
then $Z\in\mathcal{Z}$ and
\[
\degc(Z) = \mathrm{deg}((F,K),U,0)\,.
\]
\end{prop}
\begin{proof}
It is easily seen that this time $a_t$ and $b_t$ 
satisfy~\ref{uc} and~\ref{um},
for $t$ belonging to all $[0,1]$.
\par
Then the assertions follow from
Theorems~\ref{thm:controllable} and~\ref{thm:homotopydegree}.
\end{proof}
\begin{thm}
\label{thm:existence}
Let $Z\in\mathcal{Z}$ with $\degc(Z)\neq 0$.
Then $Z\neq\emptyset$.
\end{thm}
\begin{proof}
Let $U$ and $\overline{t}$ be as in~$(a)$ of
Proposition~\ref{prop:Z}.
If $Z=\emptyset$, from Theorem~\ref{thm:proper}
we infer that there exists $t\in]0,\overline{t}]$
such that~\eqref{eq:qevim} has no solution $(u,t)$
with $u\in\cl{U}$.
From Theorem~\ref{thm:existencedegree} we deduce that
\[
\degc(Z) = \mathrm{deg}((H_t,K),U,0) = 0
\]
and a contradiction follows.
\end{proof}
Along the same line, the additivity property
can be proved taking advantage
of~Theorems~\ref{thm:add-exc} and~\ref{thm:proper}.
\begin{thm}
\label{thm:additivity}
Let $Z_0, Z_1\in \mathcal{Z}$ with $Z_0\cap Z_1=\emptyset$.
Then $Z_0\cup Z_1\in \mathcal{Z}$ and
\[
\degc(Z_0\cup Z_1) =
\degc(Z_0) + \degc(Z_1)\,.
\]
\end{thm}
\begin{thm}
\label{thm:homotopy}
Let
\begin{align*}
& a:\Omega\times(\R\times\R^n\times[0,1])\rightarrow \R^n\,,\\
\noalign{\medskip}
& b:\Omega\times(\R\times\R^n\times[0,1])\rightarrow \R
\end{align*}
be two Carath\'eodory functions 
satisfying~\ref{uni} and~\ref{um}
with respect to $T=[0,1]$ and set
\[
H_t(u) = - \mathrm{div}\left[a_t(x,u,\nabla u)\right]
+ b_t(x,u,\nabla u)\,.
\]
Let also, for each $t\in [0,1]$, 
$\underline{u}_t:\Omega\rightarrow\Rb$ be a $p$-quasi upper 
semicontinuous function and $\overline{u}_t:\Omega\rightarrow\Rb$
a $p$-quasi lower semicontinuous function, define
$K_t$ as in section~\ref{sect:qevinp} and assume that:
\begin{itemize}
\item
the functions $\underline{u}_t, \overline{u}_t$ are bounded
uniformly with respect to $t\in[0,1]$, 
we have $K_t\neq\emptyset$ for any $t\in[0,1]$ 
and assumption~\ref{mc} is satisfied.
\end{itemize}
\indent
Then the following facts hold:
\begin{enumerate}
\item[$(a)$]
the set
\[
\widehat{Z}^{tot} :=
\left\{(u,t)\in [W^{1,p}_0(\Omega)\cap L^\infty(\Omega)]
\times[0,1]:\,\,\text{$(u,t)$
is a solution of~\eqref{eq:qevim}}\right\}
\]
is (strongly) compact in $W^{1,p}_0(\Omega)\times[0,1]$
(possibly empty);
\item[$(b)$]
if $\widehat{Z}$ is open and closed in $\widehat{Z}^{tot}$ 
with respect to the topology of $W^{1,p}_0(\Omega)\times[0,1]$ and
\[
\widehat{Z}_t = \left\{u\in W^{1,p}_0(\Omega)\cap L^\infty(\Omega):\,\,
(u,t)\in \widehat{Z}\right\}\,,
\]
then $\widehat{Z}_t\in \mathcal{Z}(H_t,K_t)$
for any $t\in[0,1]$ and
$\degc((H_t,K_t),\widehat{Z}_t)$ is independent of $t\in[0,1]$.
\end{enumerate}
\end{thm}
\begin{proof}
First of all, the set $\widehat{Z}^{tot}$ is compact by
Theorem~\ref{thm:proper}.
To prove assertion~$(b)$, let $W$ be a bounded and open
subset of $W^{1,p}_0(\Omega)\times[0,1]$ such that
\[
\widehat{Z} = \widehat{Z}^{tot}\cap W 
= \widehat{Z}^{tot}\cap\overline{W}\,.
\]
In particular, we have $\widehat{Z}_t\in \mathcal{Z}(H_t,K_t)$ 
for any $t\in[0,1]$.
\par
Let $\widehat{T}=[0,1]\times[0,1]$, let
\begin{align*}
&a_{t,\tau}(x,s,\xi) = \psi(\tau |s|)\,a_t(x,s,\xi)
+ [1-\psi(\tau |s|)]\,|\xi|^{p-2}\xi\,,\\
&b_{t,\tau}(x,s,\xi) = \Psi_{\tau}(b_t(x,s,\xi))\,,
\end{align*}
for $(t,\tau)\in\widehat{T}$  and define 
\begin{gather*}
H_{t,\tau}(u) =
- \mathrm{div}[a_{t,\tau}(x,u,\nabla u)] 
+ b_{t,\tau}(x,u,\nabla u)\,,\\
K_{t,\tau} = K_t\,.
\end{gather*}
It is easily seen that $a_{t,\tau}$ and $b_{t,\tau}$
satisfy~\ref{uni} and~\ref{um} with respect to $\widehat{T}$,
so that we can consider the problem
\begin{equation}
\label{eq:qevitm}
\begin{cases}
(u,(t,\tau))\in [W^{1,p}_0(\Omega)\cap L^{\infty}(\Omega)]
\times \widehat{T}\,,\\
\noalign{\medskip}
u\in K_{t,\tau}\,, \\
\noalign{\medskip}
\displaystyle{
\int_{\Omega}
\bigl[a_{t,\tau}(x,u,\nabla u)\cdot\nabla (v-u) +
b_{t,\tau}(x,u,\nabla u)\,(v-u)\bigr]\,dx \geq 0}\\
\noalign{\medskip}
\qquad\qquad\qquad\qquad\qquad\qquad\qquad\qquad\qquad\qquad
\text{for every $v\in K_{t,\tau}$}\,.
\end{cases}
\end{equation}
Since $[0,1]$ is compact, by~Theorem~\ref{thm:proper}
there exists $\overline{\tau}\in]0,1]$ 
such that~\eqref{eq:qevitm} has no solution  
$(u,(t,\tau))$ with $(u,t)\in\partial W$ and
$0\leq\tau\leq\overline{\tau}$.
\par
By Definition~\ref{defn:degree} we infer that
\[
\degc((H_t,K_t),\widehat{Z}_t)
= \mathrm{deg}((H_{t,\overline{\tau}},K_t),W_t,0) 
\qquad\text{for any $t\in[0,1]$}
\]
and the assertion follows from Theorem~\ref{thm:homotopydegree}.
\end{proof}
\begin{rem}
By Theorem~\ref{thm:homotopy} and 
Proposition~\ref{prop:consistency}, $\degc((F,K),Z)$
can be calculated also by other approximation techniques,
with respect to the one used in Definition~\ref{defn:degree}.
\end{rem}
\begin{thm}
\label{thm:tot}
If $K\neq \emptyset$, then $\degc(Z^{tot}) = 1$.
\end{thm}
\begin{proof}
Define, for $0\leq t\leq 1$,
\begin{alignat*}{3}
&a_t(x,s,\xi) &&= 
t\,a(x,s,\xi) + (1-t)\,|\xi|^{p-2}\xi\,,\\
&b_t(x,s,\xi) &&= t\,b(x,s,\xi) \,.
\end{alignat*}
It is easily seen that $a_t$ and $b_t$ satisfy
assumptions~\ref{uni} and~\ref{um}, so that
Theorem~\ref{thm:homotopy} can be applied.
If we take $\widehat{Z} = \widehat{Z}^{tot}$, we get
\[
\degc(Z^{tot}) =
\degc((H_1,K),\widehat{Z}_1) = 
\degc((H_0,K),\widehat{Z}_0) \,.
\]
Let $u_0\in K$ and let
\[
U = \left\{u\in W^{1,p}_0(\Omega):\,\,
\|\nabla u\|_2 < r\right\}\,,
\]
with $r$ large enough to guarantee that $u_0\in U$,
$\widehat{Z}_0\subseteq U$ and
\[
\int_\Omega |\nabla v|^{p-2}\nabla v\cdot\nabla(v-u_0)\,dx
\geq 0
\qquad\text{for every $v\in \partial U$}\,.
\]
From Proposition~\ref{prop:consistency}
and Theorem~\ref{thm:deg1} we infer that
\[
\degc((H_0,K),\widehat{Z}_0) =
\mathrm{deg}((H_0,K),U,0) = 1
\]
and the assertion follows.
\end{proof}
\begin{prop}
Let $Z\in\mathcal{Z}(F,K)$ and let $\varphi\in\Phi$ and 
$\vartheta\in\Theta$.
Then $Z^\varphi\in\mathcal{Z}(F^\varphi,K^\varphi)$, 
$Z\in\mathcal{Z}(F_\vartheta,K)$ and
\[
\degc((F,K),Z) = \degc((F^\varphi,K^\varphi),Z^\varphi) =
\degc((F_\vartheta,K),Z)\,.
\]
\end{prop}
\begin{proof}
If we set
\[
\varphi_t(s) = (1-t)s + t\varphi(s)\,,\qquad
\vartheta_t(s) = (1-t) + t \vartheta(s)\,,
\]
the assertion follows from Theorem~\ref{thm:homotopy}.
\end{proof}

%--------------------------------------------------------------------

\section{Proof of Theorem~\ref{thm:main}}
We aim to apply the results of the previous sections to
\[
a(x,s,\xi) = A(x,s)\xi\,,\qquad
b(x,s,\xi) = B(x,s)|\xi|^2 - g(x,s)\,.
\]
By hypothesis~\ref{ABi}, assumptions~\ref{hni} 
and~\ref{hnii} are satisfied with $p=2$.
Moreover, if $\underline{M}$ and $\overline{M}$ are as in
hypothesis~\ref{ABii}, then $\underline{u}=\underline{M}$
and $\overline{u}=\overline{M}$ satisfy
assumption~\ref{hniii}.
\par
Denote by $(\lambda_k)$, $k\geq 1$, the sequence of the
eigenvalues of~\eqref{eq:ABlin}, repeated according to
multiplicity, and set, for a matter
of convenience, $\lambda_0=-\infty$.
\par
Finally, define $F$, $K$, $Z^{tot}$, $\mathcal{Z}$ and 
$\degc(Z)$ as before, 
observe that $K\neq\emptyset$ and set
\begin{alignat*}{3}
&Z_+
&&=\left\{u\in Z^{tot}\setminus\{0\}:\,\,
\text{$u\geq 0$ a.e. in $\Omega$}\right\}\,,\\
&Z_-
&&=\left\{u\in Z^{tot}\setminus\{0\}:\,\,
\text{$u\leq 0$ a.e. in $\Omega$}\right\}\,.
\end{alignat*}
It is easily seen that
\[
a^\varphi(x,s,\xi) = A^\varphi(x,s)\xi\,,
\qquad
b^\varphi(x,s,\xi) = B^\varphi(x,s)|\xi|^2 - g^\varphi(x,s)\,.
\]
Let us also set
\begin{gather*}
A_\vartheta(x,s)=\vartheta(s) \,A(x,s)\,,\qquad
B_\vartheta(x,s)=\vartheta'(s)\,A(x,s)+\vartheta(s)\,B(x,s)\,,\\
g_\vartheta(x,s)=\vartheta(s)\,g(x,s)\,,
\end{gather*}
so that
\[
a_\vartheta(x,s,\xi)= A_\vartheta(x,s) \xi\,,\qquad
b_\vartheta(x,s,\xi)= B_\vartheta(x,s)|\xi|^2 - g_\vartheta(x,s)\,.
\]
\begin{prop}
\label{prop:sign}
For every $R>0$ there exist $\vartheta_1, \vartheta_2\in\Theta$,
depending only on $\beta_R$ and $\nu_R$, such that
\[
B_{\vartheta_1}(x,s) \leq 0\leq B_{\vartheta_2}(x,s)
\qquad\text{for a.e. $x\in\Omega$ and every $s\in\R$
with $|s|\leq R$}\,.
\]
\end{prop}
\begin{proof}
It is a simple variant of Proposition~\ref{prop:signgen}.
\end{proof}
\begin{prop}
\label{prop:nat}
Let 
\[
\hat{g}:\Omega\times\R\rightarrow\R
\]
be a Carath\'eodory function such that
for every $R>0$ there exists $\beta_R > 0$ satisfying
\[
|\hat{g}(x,s)| \leq \beta_R 
\qquad\text{for a.e. $x\in\Omega$ and every 
$s\in\R$ with $|s|\leq R$}
\]
and such that
\[
\hat{g}(x,\underline{M}) \geq 0 \geq \hat{g}(x,\overline{M})
\qquad\text{for a.e. $x\in\Omega$}\,.
\]
\indent
If $u$ is a solution of the variational 
inequality~\eqref{eq:qevi} with
\[
a(x,s,\xi) = A(x,s)\xi\,,\qquad
b(x,s,\xi) = B(x,s)|\xi|^2 - \hat{g}(x,s)\,,
\]
then $u$ satisfies the equation
\begin{multline*}
\int_{\Omega}
\bigl[A(x,u)\nabla u\cdot\nabla v +
B(x,u)\,|\nabla u|^2\,v\bigr]\,dx \\
=
\int_\Omega \hat{g}(x,u)\,v\,dx
\qquad\text{for any 
$v\in W^{1,2}_0(\Omega)\cap L^\infty(\Omega)$}\,.
\end{multline*}
\end{prop}
\begin{proof}
Let $v\in W^{1,2}_0(\Omega)\cap L^\infty(\Omega)$
with $v\geq 0$ a.e. in $\Omega$, let $t>0$ and let
\[
u_t = \min\{u+tv,\overline{M}\}\,.
\]
Since $u_t\in K$, it follows
\begin{multline*}
\frac{1}{t}\,\int_{\Omega}
A(x,u)\nabla u\cdot\nabla (u_t-u)\,dx \\
\geq - \int_{\Omega}
B(x,u)\,|\nabla u|^2\,\frac{u_t-u}{t}\,dx +
\int_\Omega \hat{g}(x,u)\,\frac{u_t-u}{t}\,dx \\
= - \int_{\{u<\overline{M}\}}
B(x,u)\,|\nabla u|^2\,\frac{u_t-u}{t}\,dx +
\int_{\{u<\overline{M}\}} 
\hat{g}(x,u)\,\frac{u_t-u}{t}\,dx 
 \,.
\end{multline*}
On the other hand, we have
\begin{multline*}
\frac{1}{t}\,\int_{\Omega}
A(x,u)\nabla u\cdot\nabla (u_t-u)\,dx \\
= \int_{\{u+tv < \overline{M}\}}
A(x,u)\nabla u\cdot\nabla v\,dx
- \frac{1}{t}\,\int_{\{u+tv\geq \overline{M}\}}
A(x,u)|\nabla u|^2\,dx \\
\leq
\int_{\{u+tv < \overline{M}\}}
A(x,u)\nabla u\cdot\nabla v\,dx \,,
\end{multline*}
whence
\begin{multline*}
\int_{\{u+tv < \overline{M}\}}
A(x,u)\nabla u\cdot\nabla v\,dx \\
\geq
- \int_{\{u<\overline{M}\}}
B(x,u)\,|\nabla u|^2\,\frac{u_t-u}{t}\,dx +
\int_{\{u<\overline{M}\}} 
\hat{g}(x,u)\,\frac{u_t-u}{t}\,dx \,.
\end{multline*}
Since $0\leq u_t - u \leq tv$, we can go to the
limit as $t\to 0^+$, obtaining
\[
\begin{split}
\int_{\Omega}
A(x,u)\nabla u\cdot\nabla v\,dx &=
\int_{\{u < \overline{M}\}}
A(x,u)\nabla u\cdot\nabla v\,dx \\
&\geq
- \int_{\{u<\overline{M}\}}
B(x,u)\,|\nabla u|^2\,v\,dx +
\int_{\{u<\overline{M}\}} 
\hat{g}(x,u)\,v\,dx \\
&= - \int_{\Omega}
B(x,u)\,|\nabla u|^2\,v\,dx +
\int_{\Omega} \hat{g}(x,u)\,v\,dx \\
&\qquad\qquad\qquad\qquad\qquad\qquad\qquad
- \int_{\{u=\overline{M}\}} 
\hat{g}(x,\overline{M})\,v\,dx \\
&\geq
- \int_{\Omega}
B(x,u)\,|\nabla u|^2\,v\,dx +
\int_{\Omega} \hat{g}(x,u)\,v\,dx \,.
\end{split}
\]
Arguing on $u_t=\max\{u-tv,\underline{M}\}$,
one can prove in a similar way that
\[
\int_{\Omega}
A(x,u)\nabla u\cdot\nabla v\,dx \leq
- \int_{\Omega}
B(x,u)\,|\nabla u|^2\,v\,dx +
\int_{\Omega} \hat{g}(x,u)\,v\,dx \,,
\]
whence
\begin{multline*}
\int_{\Omega}
\bigl[A(x,u)\nabla u\cdot\nabla v +
B(x,u)\,|\nabla u|^2\,v\bigr]\,dx =
\int_\Omega \hat{g}(x,u)\,v\,dx\\
\text{for any 
$v\in W^{1,2}_0(\Omega)\cap L^\infty(\Omega)$ with
$v\geq 0$ a.e. in $\Omega$}
\end{multline*}
and the assertion follows.
\end{proof}
\begin{prop}
\label{prop:strongmax}
Let $\Omega$ be connected and assume that
$u\in W^{1,2}_0(\Omega)\cap L^\infty(\Omega)$
satisfies $u\geq 0$ a.e. in $\Omega$, $u>0$ on a set
of positive measure and
\begin{multline*}
\int_{\Omega}
\bigl[A(x,u)\nabla u\cdot\nabla v +
B(x,u)\,|\nabla u|^2\,v\bigr]\,dx \\
\geq
\int_\Omega g(x,u)\,v\,dx
\qquad\text{for any 
$v\in W^{1,2}_0(\Omega)\cap L^\infty(\Omega)$
with $v\geq 0$ a.e. in $\Omega$}\,.
\end{multline*}
\indent
Then we have 
\[
\text{$\displaystyle{\essinf_C\,u >0}$
for every compact subset $C$ of $\Omega$}\,.
\]
\end{prop}
\begin{proof}
By Propositions~\ref{prop:theta} and~\ref{prop:sign},
we may assume without loss of generality that
\[
B(x,s)\leq 0
\qquad\text{for a.e. $x\in\Omega$ and every $s\in\R$
with $|s|\leq\|u\|_\infty$}\,.
\]
Then for every $v\in W^{1,2}_0(\Omega)\cap L^\infty(\Omega)$
with $v\geq 0$ a.e. in $\Omega$, we have
\[
\int_{\Omega}
A(x,u)\nabla u\cdot\nabla v\,dx \geq
\int_{\Omega} g(x,u)\,v\,dx =
\int_{\Omega} \gamma(x,u)\,uv\,dx
\]
with $\gamma(x,u)\in L^\infty(\Omega)$.
\par
From~\cite[Theorem~8.15 and 
Remark~8.16]{motreanu_motreanu_papageorgiou2014}
the assertion follows.
\end{proof}
\begin{lem}
\label{lem:pos}
Assume that $\Omega$ is connected and that $\lambda_1<0$.
Moreover, according to~\ref{ABi}, let $\beta > 0$ be such that
\[
|g(x,s)| \leq \beta|s|
\qquad\text{for a.e. $x\in\Omega$ and every $s\in\R$
with $\underline{M} \leq s \leq\overline{M}$}\,.
\]
Let also $\psi:\R\rightarrow[0,1]$ be a continuous function,
with $\psi(0)>0$ and $\psi(s)=0$ outside
$]\underline{M},\overline{M}[$, and consider the problem
\begin{equation}
\label{eq:pos}
\begin{cases}
(u,t)\in K\times[0,1]\,,\\
\noalign{\medskip}
\displaystyle{
\int_{\Omega}
\bigl[A(x,u)\nabla u\cdot\nabla (v-u) +
B(x,u)\,|\nabla u|^2\,(v-u)\bigr]\,dx \geq 
\int_\Omega g_t(x,u)(v-u)}\\
\noalign{\medskip}
\qquad\qquad\qquad\qquad\qquad\qquad\qquad\qquad\qquad\qquad
\text{for every $v\in K$}\,,
\end{cases}
\end{equation}
where
\[
g_t(x,s) = g(x,s) + t(\psi(s) + \beta\,s^-)\,.
\] 
Denote by $\widehat{Z}^{tot}$ the set of solutions $(u,t)$
of~\eqref{eq:pos} and let
\[
\widehat{Z} =\left\{(u,t)\in \widehat{Z}^{tot}:\,\,
\text{$u\geq 0$ a.e. in $\Omega$ and $u>0$ on a set
of positive measure}\right\}\,.
\]
\indent
Then there exist $0<r_1<r_2$ such that
\begin{multline*}
\widehat{Z}
= \left\{(u,t)\in \widehat{Z}^{tot}:\,\,
\int_\Omega (u^-)^2\,dx < r_1^2 <r_2^2 <
\int_\Omega (u^+)^2\,dx \right\} \\
= \left\{(u,t)\in \widehat{Z}^{tot}:\,\,
\int_\Omega (u^-)^2\,dx \leq r_1^2 <r_2^2 \leq
\int_\Omega (u^+)^2\,dx \right\} \,.
\end{multline*}
\end{lem}
\begin{proof}
If we set
\[
a_t(x,s,\xi) = A(x,s)\xi\,,\qquad
b_t(x,s,\xi) = B(x,s)|\xi|^2 - g_t(x,s)\,,
\]
it is easily seen that assumptions~\ref{uni} and~\ref{um}
are satisfied.
From Theorem~\ref{thm:homotopy} we infer that
$\widehat{Z}^{tot}$ is compact in 
$W^{1,p}_0(\Omega)\times[0,1]$.
\par
First of all, we claim that there exists $r_2>0$ such that
\[
\widehat{Z}
\subseteq \left\{(u,t)\in \widehat{Z}^{tot}:\,\,
r_2^2 < \int_\Omega (u^+)^2\,dx \right\}\,.
\]
By Propositions~\ref{prop:theta} and~\ref{prop:sign},
we may assume without loss of generality that
\[
B(x,s)\geq 0
\qquad\text{for a.e. $x\in\Omega$ and every $s\in\R$
with $\underline{M}\leq s\leq\overline{M}$}\,.
\]
Assume, for a contradiction, that $(u_k,t_k)$ is a sequence
in $\widehat{Z}$ with $\|u_k\|_2\to 0$.
Then we may suppose, 
without loss of generality, that $(u_k)$ is convergent
to $0$ in $W^{1,p}_0(\Omega)$ and a.e. in $\Omega$
and that $(t_k)$ is convergent to some $t\in [0,1]$.
Let $u_k=\tau_k z_k$ with $\tau_k=\|\nabla u_k\|_2$ and,
up to a subsequence, $(z_k)$ weakly convergent to some
$z$ in $W^{1,2}_0(\Omega)$.
\par
If $v\in K\setminus\{0\}$ with $v\geq 0$ a.e. in $\Omega$,
we have
\begin{multline*}
\int_{\Omega}
\bigl[A(x,u_k)\nabla z_k\cdot\nabla (v-u_k) +
\tau_k\,B(x,u_k)\,|\nabla z_k|^2\,
(v-u_k)\bigr]\,dx \\
\geq 
\int_\Omega \frac{g(x,\tau_k z_k)}{\tau_k} 
\,(v-u_k)\,dx 
+ \frac{t_k}{\tau_k}\,
\int_\Omega \psi(u_k)(v-u_k)\,dx\,,
\end{multline*}
which implies that $(t_k/\tau_k)$ is bounded hence convergent,
up to a subsequence, to some $\sigma\geq 0$.
\par
Then we also get
\[
\int_{\Omega}
A(x,0)\nabla z\cdot\nabla v\,dx
\geq 
\int_\Omega D_sg(x,0)zv\,dx 
+ \sigma\,\psi(0) \int_\Omega v\,dx
\qquad\text{for any $v\in K$}\,,
\]
whence
\begin{multline*}
\int_{\Omega}
\left[A(x,0)\nabla z\cdot\nabla v - D_sg(x,0)zv\right]\,dx 
= \sigma\,\psi(0) \int_\Omega v\,dx \\
\qquad\text{for any 
$v\in W^{1,2}_0(\Omega)\cap L^\infty(\Omega)$}\,.
\end{multline*}
If we choose $v=\varphi_1$, where $\varphi_1$ is a positive 
eigenfunction of~\eqref{eq:ABlin} associated with $\lambda_1<0$,
we get
\[
\lambda_1 \,\int_{\Omega} z \varphi_1 \,dx
=
\sigma\psi(0)\,\int_\Omega \varphi_1\,dx\,,
\]
whence $z=0$.
\par
Finally, the choice $v=0$ in~\eqref{eq:pos} yields
\begin{multline*}
\int_{\Omega}
A(x,u_k)|\nabla z_k|^2 \,dx \leq
\int_{\Omega}
\bigl[A(x,u_k)|\nabla z_k|^2 +
\tau_k\,B(x,u_k)\,|\nabla z_k|^2\,
z_k\bigr]\,dx \\
\leq 
\int_\Omega \frac{g(x,\tau_k z_k)}{\tau_k} 
\,z_k\,dx 
+ \frac{t_k}{\tau_k}\,
\int_\Omega \psi(u_k)z_k\,dx\,.
\end{multline*}
We infer that $\|\nabla z_k\|_2 \to 0$
and a contradiction follows.
\par
With this choice of $r_2$, we also have
\begin{multline*}
\widehat{Z}
\subseteq \left\{(u,t)\in \widehat{Z}^{tot}:\,\,
\int_\Omega (u^-)^2\,dx < r_1^2 <r_2^2 <
\int_\Omega (u^+)^2\,dx \right\} \\
\subseteq \left\{(u,t)\in \widehat{Z}^{tot}:\,\,
\int_\Omega (u^-)^2\,dx \leq r_1^2 <r_2^2 \leq
\int_\Omega (u^+)^2\,dx \right\} 
\end{multline*}
for every $r_1\in]0,r_2[$.
Now we claim that there exists $r_1\in]0,r_2[$ such that
\[
\left\{(u,t)\in \widehat{Z}^{tot}:\,\,
\int_\Omega (u^-)^2\,dx \leq r_1^2 <r_2^2 \leq
\int_\Omega (u^+)^2\,dx \right\} 
\subseteq
\widehat{Z} \,.
\]
By Propositions~\ref{prop:theta} and~\ref{prop:sign},
now we may assume without loss of generality that
\[
B(x,s)\leq 0
\qquad\text{for a.e. $x\in\Omega$ and every $s\in\R$
with $\underline{M}\leq s\leq\overline{M}$}\,.
\]
Assume, for a contradiction, that $(u_k,t_k)$ is a 
sequence in the set at the left hand side
with $(u_k,t_k)\not\in \widehat{Z}$
and $\|u_k^-\|_2\to 0$.
Then, up to a subsequence, $(u_k)$ is convergent
to some $u\in\widehat{Z}$ in $W^{1,p}_0(\Omega)$ 
and a.e. in $\Omega$, while
$(t_k)$ is convergent to some $t\in [0,1]$.
By Propositions~\ref{prop:nat} and~\ref{prop:strongmax}
we have $u>0$ a.e. in $\Omega$.
If we write $u_k^-=\tau_k z_k$ with
$\tau_k=\|\nabla u_k^-\|_2$, we have that
$(z_k)$ is weakly convergent, up to a subsequence,
to some $z$ in $W^{1,2}_0(\Omega)$ and, on the other
hand, $(z_k)$ is convergent to $0$ a.e. in $\Omega$,
as $u>0$.
The choice $v=u_k^+$ in~\eqref{eq:pos} yields
\begin{multline*}
\int_{\Omega}
\bigl[A(x,u_k)\nabla u_k\cdot\nabla u_k^- +
\,B(x,u_k)\,|\nabla u_k|^2\,
u_k^-\bigr]\,dx \\
\geq 
\int_\Omega g(x,u_k) 
\,u_k^-\,dx 
+ t_k\,\int_\Omega \left[\psi(u_k)+\beta u_k^-\right]
u_k^-\,dx\,,
\end{multline*}
whence
\[
- \int_{\Omega}
A(x,u_k)|\nabla z_k|^2\,dx 
\geq 
\int_\Omega \frac{g(x,-\tau_k z_k)}{\tau_k}\,z_k\,dx \,.
\]
We infer that $\|\nabla z_k\|_2\to 0$
and a contradiction follows.
\end{proof}
\begin{prop}
\label{prop:Zpm}
If $\Omega$ is connected and $\lambda_1<0$, 
then we have $Z_+, Z_-\in\mathcal {Z}$ and
\[
\degc\left(Z_+\right)=\degc\left(Z_-\right) =1\,.
\]
\end{prop}
\begin{proof}
Let $g_t$, $\widehat{Z}^{tot}$, $\widehat{Z}$,
$r_1$ and $r_2$ be as in Lemma~\ref{lem:pos}.
We aim to apply again Theorem~\ref{thm:homotopy} with
\[
a_t(x,s,\xi) = A(x,s)\xi\,,\qquad
b_t(x,s,\xi) = B(x,s)|\xi|^2 - g_t(x,s)\,.
\]
By Lemma~\ref{lem:pos} the set $\widehat{Z}$
is open and closed in $\widehat{Z}^{tot}$.
From Theorem~\ref{thm:homotopy} we infer that
\[
Z_+ = \widehat{Z}_0 \in \mathcal{Z}(H_0,K) = \mathcal{Z}
\]
and that
\[
\degc\left(Z_+\right) =
\degc\left((H_0,K),\widehat{Z}_0\right) = 
\degc\left((H_1,K),\widehat{Z}_1\right)\,.
\]
Now we claim that $\widehat{Z}_1=Z^{tot}(H_1,K)$.
By Propositions~\ref{prop:theta} and~\ref{prop:sign}
we may assume without loss of generality that
\[
B(x,s) \leq 0
\qquad\text{for a.e. $x\in\Omega$ and every $s\in\R$
with $\underline{M} \leq s \leq\overline{M}$}\,.
\]
If we take $v\in K\setminus\{0\}$ with $v\geq 0$ 
a.e. in $\Omega$ in~\eqref{eq:pos}, 
we see that $0\not\in Z^{tot}(H_1,K)$.
Moreover, if $u\in Z^{tot}(H_1,K)$,
the choice $v=u^+$ yields $(v-u)=u^-$, hence
\begin{multline*}
- \int_{\Omega} A(x,u)|\nabla u^-|^2\,dx \geq
\int_{\Omega}
\bigl[A(x,u)\nabla u\cdot\nabla u^- +
B(x,u)\,|\nabla u|^2\,u^-\bigr]\,dx \\
\geq 
\int_\Omega g_{1}(x,u)u^-\,dx \geq 0\,.
\end{multline*}
Therefore $u^-=0$, whence $u\in \widehat{Z}_1$
and the claim is proved.
\par
From Theorem~\ref{thm:tot} we infer that
\[
\degc\left((H_1,K),\widehat{Z}_1\right)=
\degc\left((H_1,K),Z^{tot}(H_1,K)\right)=1
\]
and the assertion concerning $Z_+$ follows.
\par
The assertion concerning $Z_-$ can be proved in a similar way.
\end{proof}
\begin{prop}
\label{prop:zero}
If there exists $k\geq 0$ with $\lambda_k<0<\lambda_{k+1}$, 
then $\{0\}\in\mathcal{Z}$ and
\[
\degc\left(\{0\}\right)=(-1)^k\,.
\]
\end{prop}
\begin{proof}
If we set
\[
A_t(x,s) = A(x,ts)\,,\qquad
B_t(x,s) = t\,B(x,ts)\,,
\]
\[
g_t(x,s) = 
\left\{
\begin{array}{ll}
\displaystyle{
\frac{g(x,ts)}{t}}
&\qquad\text{if $0<t\leq 1$}\,,\\
\noalign{\medskip}
D_sg(x,0)s
&\qquad\text{if $t=0$}\,,
\end{array}
\right.
\]
\[
a_t(x,s,\xi) = A_t(x,s)\xi\,,\qquad
b_t(x,s,\xi) = B_t(x,s)|\xi|^2 - g_t(x,s)\,,
\]
it is easily seen that assumptions~\ref{uni} 
and~\ref{um} are satisfied.
We aim to apply Theorem~\ref{thm:homotopy}.
\par
We claim that there exists $r>0$ such that,
if $(u,t)\in\widehat{Z}^{tot}$ and
$\|\nabla u\|_2\leq r$, then $u=0$.
Assume, for a contradiction, that $(u_k,t_k)$
is a sequence in $\widehat{Z}^{tot}$ with $u_k\neq 0$ and
$\|\nabla u_k\|_2\to 0$.
Let $u_k=\tau_k z_k$ with $\tau_k=\|\nabla u_k\|_2$
and, up to a subsequence, $(z_k)$ weakly convergent
to some $z$ in $W^{1,2}_0(\Omega)$ and $(t_k)$
convergent to some $t$ in $[0,1]$.
\par
Given $v\in K$, if $t_k>0$ we have
\begin{multline}
\label{eq:zero}
\int_{\Omega}
\bigl[A(x,t_k u_k)\nabla u_k\cdot\nabla (v-u_k) +
t_k\,B(x,t_k u_k)\,|\nabla u_k|^2\,
(v-u_k)\bigr]\,dx \\
\geq 
\int_\Omega \frac{g(x,t_k u_k)}{t_k} 
\,(v-u_k)\,dx \,,
\end{multline}
whence
\begin{multline*}
\int_{\Omega}
\bigl[A(x,t_k u_k)\nabla z_k\cdot\nabla (v-u_k) +
\tau_k t_k\,B(x,t_k u_k)\,|\nabla z_k|^2\,
(v-u_k)\bigr]\,dx \\
\geq 
\int_\Omega \frac{g(x,\tau_k t_k z_k)}{\tau_k t_k} 
\,(v-u_k)\,dx \,.
\end{multline*}
Going to the limit as $k\to\infty$, we get
\[
\int_{\Omega}
A(x,0)\nabla z\cdot\nabla v \,dx \\
\geq 
\int_\Omega D_sg(x,0) zv\,dx 
\qquad\text{for any $v\in K$}\,.
\]
If $t_k=0$, we have
\[
\int_{\Omega}
A(x,0)\nabla z_k\cdot\nabla (v-u_k) \,dx \\
\geq 
\int_\Omega D_sg(x,0) z_k(v-u_k)\,dx 
\]
and the same conclusion easily follows.
\par
Then we infer that
\[
\int_{\Omega}
A(x,0)\nabla z\cdot\nabla v \,dx \\
=
\int_\Omega D_sg(x,0) zv\,dx 
\qquad\text{for any 
$v\in W^{1,2}_0(\Omega)\cap L^\infty(\Omega)$}\,,
\]
whence $z=0$, as $0$ is not in the sequence 
$(\lambda_k)$.
\par
By Propositions~\ref{prop:theta} and~\ref{prop:sign},
we may assume without loss of generality that
\[
B(x,s) \geq 0
\qquad\text{for a.e. $x\in\Omega$ and every
$s\in\R$ with $\underline{M}\leq s\leq \overline{M}$}\,.
\]
If $t_k>0$, the choice $v=-u_k^-$ in~\eqref{eq:zero} yields
\begin{multline*}
\int_{\Omega}
A(x,t_k u_k)|\nabla z_k^+|^2\,dx \leq
\int_{\Omega}
\bigl[A(x,t_ku_k)|\nabla z_k^+|^2 +
\tau_k t_k\,B(x,t_k u_k)\,|\nabla z_k|^2\,
z_k^+\bigr]\,dx \\
\leq 
\int_\Omega \frac{g(x,\tau_k t_k z_k)}{\tau_k t_k} 
\,z_k^+\,dx \,,
\end{multline*}
which implies that $\|\nabla z_k^+\|_2\to 0$.
If $t_k=0$ the argument is analogous and simpler.
In a similar way one can show that
$\|\nabla z_k^-\|_2\to 0$ and a contradiction follows.
Therefore, there exists $r>0$ with the required property.
\par
In particular, we can apply Theorem~\ref{thm:homotopy} with
\[
\widehat{Z} = \{0\}\times[0,1]\,,
\]
obtaining
\[
\{0\} = \widehat{Z}_1 \in \mathcal{Z}(H_1,K) = \mathcal{Z}
\]
and
\[
\degc\left(\{0\}\right) =
\degc\left((H_1,K),\widehat{Z}_1\right) = 
\degc\left((H_0,K),\widehat{Z}_0\right) =
\degc\left((H_0,K),\{0\}\right) \,.
\]
On the other hand, if we set
\[
U = \left\{u\in W^{1,2}_0(\Omega):\,\,\|\nabla u\|_2<r\right\}\,,
\]
from Proposition~\ref{prop:consistency} we infer that
\[
\degc((H_0,K),\{0\}) = \mathrm{deg}((H_0,K),U,0)\,.
\]
Now, if we set
\[
K_t = \left\{
\begin{array}{ll}
\frac{1}{t}\,K
&\qquad\text{if $0<t\leq 1$}\,,\\
\noalign{\medskip}
W^{1,2}_0(\Omega)
&\qquad\text{if $t=0$}\,,
\end{array}
\right.
\]
from~\cite[Theorem~4.53 
and~Proposition~4.61]{motreanu_motreanu_papageorgiou2014}
we deduce that
\[
\mathrm{deg}((H_0,K),U,0) =
\mathrm{deg}((H_0,W^{1,2}_0(\Omega)),U,0)\,.
\]
Finally, from~\cite[Theorem~2.5.2]{skrypnik1994} it 
follows that 
\[
\mathrm{deg}((H_0,W^{1,2}_0(\Omega)),U,0) = (-1)^k \,.
\]
\end{proof}
\par\bigskip
\noindent
\emph{Proof of Theorem~\ref{thm:main}.}
\par\noindent
From Proposition~\ref{prop:Zpm} we know that
\[
\degc\left(Z_+\right)=\degc\left(Z_-\right) =1 \,.
\]
By Theorem~\ref{thm:existence} and Propositions~\ref{prop:nat}
and~\ref{prop:strongmax}
we infer that there exist at least two solutions 
$u_1\in Z_-$ and $u_2\in Z_+$ of~\eqref{eq:AB}
with
\[
\text{$\displaystyle{\esssup_C\,u_1 < 0 <\essinf_C\,u_2}$\quad
for every compact subset $C$ of $\Omega$}\,.
\]
Assume, for a contradiction, that
\[
Z^{tot} = Z_- \cup \{0\} \cup Z_+
\]
with $\degc\left(\{0\}\right)=1$ by Proposition~\ref{prop:zero}.
From Theorems~\ref{thm:tot} and~\ref{thm:additivity} 
we infer that
\[
1 = \degc\left(Z^{tot}\right) = 3
\]
and a contradiction follows.
Therefore there exists 
\[
u_3\in Z^{tot} \setminus\left(Z_- \cup \{0\} \cup Z_+\right)\,.
\]
By Proposition~\ref{prop:nat} $u_3$ is a sign-changing solution
of~\eqref{eq:AB}.
According to~\cite[Theorem~VII.1.1]{giaquinta1983}, each $u_j$ 
is locally H\"{o}lder continuous in $\Omega$. 
\qed

%--------------------------------------------------------------------

%
\end{document}